\documentclass[a4paper,12pt]{article}
\usepackage[utf8]{inputenc}
\usepackage[english]{babel}
\usepackage{amsthm,amsfonts}
\usepackage{hyperref}
\usepackage[usenames,dvipsnames]{xcolor}
\usepackage{amsmath,mathrsfs}
\usepackage{amsfonts}
\usepackage[all]{xy}
\usepackage{graphicx}
\usepackage{enumerate}
\usepackage{mathabx}
\hypersetup{
    colorlinks=true,
    linkcolor=blue,
    filecolor=magenta,      
    urlcolor=cyan,
}

\def\ZZ{\mathbb{Z}}
  \def\r{\mathfrak{r}}
  \def\D{\mathcal{D}}
  \def\Z{\mathbb{Z}}
  \def\H{\mathbb{H}}
  \def\M{\mathcal{M}}

  \def\stab{\mathrm{Stab}}

  \def\U{\mathcal U}

  \def\GG{\SO_0(1,2)}
  \def\isom{Isom}
  \def\lsem{[\![}
  \def\rsem{]\!]}
\def\V{\mathcal V}

\def\C{\mathcal C}
\def\R{\mathbb{R}}
\def\E{\mathbf E}

\def\S{\mathcal S}
\def\N{\mathbb N}
\def\T{\mathcal{T}}

\def\d{\mathrm{d}}
\def\S{\mathcal{S}}

\def\D{\mathcal{D}}
\def\SO{\mathrm{SO}}
\def\E{\mathbb{E}}

\def\sing{\mathrm{Sing}}
\def\RR{\mathbb{R}}
\newcommand{\mass}[1]{{\mathbb E}^{1,2}_{#1}}
\def\isom{\mathrm{Isom}}

\def\Hol{\mathrm{Hol}}
\def\Reg{\mathrm{Reg}}
\def\fix{\mathrm{Fix}}

\def\BTZext{\mathrm{Ext}_{BTZ}}
\def\pgcd{M_1\wedge M_2}
\def\ppcm{M_1\vee M_2}
\def\Stilde{\widetilde{\mathcal S}}
\def\reg{\Reg}
\def\P{\mathcal P}

\newcommand{\fonction}[5]{\displaystyle#1:\begin{array}{l|rcl}
& \displaystyle #2 & \longrightarrow & \displaystyle #3 \\
    & \displaystyle #4 & \longmapsto & \displaystyle #5 \end{array}}
\newcommand{\fonctionn}[4]{\displaystyle\begin{array}{l|rcl}
& \displaystyle #1 & \longrightarrow & \displaystyle #2 \\
    & \displaystyle #3 & \longmapsto & \displaystyle #4 \end{array}}

\definecolor{MyGreen}{rgb}{0.0,.5,0.0}

\definecolor{MyDarkRed}{rgb}{0.7,0,0}

\usepackage{float}

\title{Cauchy-compact flat spacetimes with extreme~BTZ}
\author{L\'eo Brunswic}

\newtheorem{theo}{Theorem}    

\newtheorem{lemou}{Lemma}[section]          
\newtheorem{propou}[lemou]{Proposition}
\newtheorem{corou}[lemou]{Corollary}

\newtheorem{defou}[lemou]{Definition}    
\newtheorem{remark}{Remark}             
\begin{document}

\title{Cauchy-compact flat spacetimes with extreme~BTZ}


\author{L\'eo Brunswic}


\maketitle

\begin{abstract}

Cauchy-compact flat spacetimes with extreme BTZ are Lorentzian analogue of complete hyperbolic surfaces of finite volume. Indeed, the latter are 2-manifolds locally modeled on the hyperbolic plane, with group of isometries $\mathrm{PSL}_2(\R)$, admitting finitely many cuspidal ends while the regular part of the former are 3-manifolds locally models on 3 dimensionnal Minkowski 
space, with group of isometries 
$\mathrm{PSL}_2(\R)\ltimes \RR^3$,
admitting finitely many ends whose neighborhoods are foliated by cusps. We prove a Theorem akin to the classical parametrization result for finite volume complete hyperbolic surfaces: the tangent bundle of the Teichmüller space of a punctured surface parametrizes globally hyperbolic Cauchy-maximal and Cauchy-compact locally Minkowski manifolds with extreme BTZ. Previous results of Mess, Bonsante and Barbot provide already a satisfactory parametrization of regular parts of such manifolds, the particularity of the present work reside in the consideration of manifolds with a singular geometrical structure with a singularities modeled on extreme BTZ. 
We present a BTZ-extension procedure akin to the procedure compactifying finite volume complete hyperbolic surface by adding cusp points at infinity
;

\end{abstract}

\maketitle

\tableofcontents
\section{Introduction}
\subsection{Context and main result}
Let $\mass{}$ be the 3-dimensional Minkowski space, ie the oriented and time-oriented affine space $\RR^3$ endowed with  the quadratic form $T(t,x,y)=-t^2+x^2+y^2$; 
the time orientation being defined as follow. A non zero vector  $u$ of $\mass{}$ is called {\it spacelike} (resp. {\it timelike}, resp. {\it lightlike}) if $T(u)>0$ (resp. $T(u)<0$, resp. $T(u)=0$).
Furthermore, $u$ is {\it future 
causal} (resp. {\it past causal}) if it is timelike or lightlike and its $t$ coordinate positive (resp. negative).
The {\it time orientation} of Minkowski space is the pair $(\leq,\ll)$ 
where $\leq$ (resp. $\ll$) is the so-called  causal (resp. chronological) order: $p \leq q $ (resp. $p \ll q$) if $q-p$ is future causal (resp. future  timelike). 
The group of isomorphisms which preserve the orientation as well as the time-orientation of Minkowski space is the identity component $\isom_0(\mass{})\simeq \mathrm{PSL}_2(\RR)\ltimes \RR^3\simeq \SO_0(1,2)\ltimes \RR^3$ of the  group of affine isometries of $\mass{}$. 
The action of $\isom_0(\mass{})$ on $\mass{}$ being analytical in the sense of Ehresmann \cite{MR794193_loc_hom,MR957518}, one can consider $(\isom_0(\mass{}),\mass{})$-manifolds (which for brevity sake, we shall  write $\mass{}$-manifolds instead). The present work is devoted to the study of certain classes of singular $\mass{}$-manifolds  we will now describe more precisely.

\subsubsection*{Geometrical structures}
The corner stone of the theory of $(G,X)$-manifolds is the existence of a developing map and of the holonomy map: under mild assumptions on $G$ and $X$ -- that are satisfied by  $(\isom_0(\mass{}),\mass{})$ -- and given a $(G,X)$-manifold $M$,
 there exists a $(G,X)$-morphism $\D$ from the universal cover $\widetilde{M}$ of $M$ to the model space 
 $X$ which is equivariant with respect to a morphism $\rho$ from the fundamental group or $M$, denoted $\pi_1(M)$, 
 to the group $G$ (the couple $(\D,\rho)$ is essentially unique). $\D$ is called the developing 
 map and $\rho$ the holonomy. When $X$ is a simply connected Riemannian manifold and $M$ is metrically complete, 
 then  $\D$ is automatically an isomorphism and the image of $\rho$ is discrete; therefore $M \simeq \Gamma\backslash X$ for some $\Gamma$ discrete subgroup of $G$.
This result is very efficient to reduce a geometrical problem (say classifying compact/complete locally hyperbolic $n$-manifolds) to a more algebraic one (say finding discrete subgroups of the group of isometry of the hyperbolic space of dimension $n$). One does not have such a result for Lorentzian, or more generally affine $(G,X)$-manifolds, such as $\mass{}$-manifolds: metric completeness makes no sense, compactness hypothesis is too strong for physics related purposes and $\D$ might not even be injective. To get an equivalent result, one has to translate metric completeness into a Lorentzian equivalent: the causal structure  comes into play.

\subsubsection*{Causal conditions and Mess Theorem}

A curve in a $\mass{}$-manifold $M$ is said future causal (resp. timelike)  if it is  locally increasing for the order $\leq$ (resp. $\ll$) of $\mass{}$ in charts of the $\mass{}$-atlas of $M$. 
We then extend both the causal/chronological orders of $\mass{}$ to $M$ by saying that $p\leq q$ (resp. $p\ll q$) if there exists a future causal (resp. timelike) curve in $M$ from $p$ to $q$. 
However, the relations $\leq$ and $\ll$ on $M$ may not even be order relations, one has a so-called causal hierarchy \cite{MR0469146,MR2436235} of properties of $\ll$ and $\leq$. 
Three levels of this hierarchy are of particular interest: causality, strong causality and global hyperbolicity. 
\begin{itemize}
\item $M$ is  {\it causal} if $\leq$ is an order relation 
\item $M$ is {\it globally hyperbolic} if there exists a topological hypersurface $\Sigma$ of $M$ which intersects exactly once every inextendible causal curves. 
\item $M$ is strongly causal if its topology admits a basis of causally convex domains; a domain $\U$ being causally convex in $M$ if all causal segment in $M$ whose extremities are in $\U$ is entirely in $\U$.
\end{itemize}
 Bernal and Sanchez \cite{sanchez_smooth} proved that, if $M$ is  globally hyperbolic, such a $\Sigma$ can be chosen smooth and {\it spacelike} (all tangent vectors are spacelike); in this case, the Lorentzian metric of $M$ induces a Riemannian metric on $\Sigma$. Then, a $\mass{}$-manifold is called Cauchy-compact (resp. Cauchy-complete) if there exists a smooth spacelike Cauchy-surface $\Sigma$ in $M$ which is compact (resp. metrically complete). Note that if any Cauchy-surface is compact, then all Cauchy-surfaces are compact; however, a Cauchy-complete $\mass{}$-manifold may admit non complete Cauchy-surfaces. Two fundamental theorems help clarify the picture:

\begin{itemize}
\item  Geroch \cite{MR0270697} proved that a globally hyperbolic {\it Lorentzian} manifold with Cauchy-surface $\Sigma$ is homeomorphic to $\Sigma\times \RR$
\item Choquet-Bruhat and Geroch \cite{MR0250640,sbierski_geroch}, proved that for globally hyperbolic Lorentzian manifold satisfying  Einstein's Equation (which  could be written $\mathrm{Ric}=0$ in our context), there exists an isometric embedding $\iota: M \rightarrow \overline M$ where $\overline M$ is a globally hyperbolic Lorentzian manifold satisfying the same Einstein's equation, $\iota$ sends Cauchy-surfaces of $M$ to Cauchy-surfaces of $\overline M$ and 
$(\iota,\overline M)$ is maximal among such embeddings. Such an embedding is called a {\it Cauchy-embedding}, $(\iota,\overline M)$ is unique up to isomorphisms and called the {\it Cauchy-maximal} extension of $M$. 
\end{itemize}

Mess \cite{mess,andersson:hal-00642328} noticed that if $M$ a globally hyperbolic Cauchy-compact $\mass{}$-manifold of Cauchy surface $\Sigma$, then its developing map $\D$ is an embedding and $\rho$ is discrete. This remark allows to try a  procedure similar to the one  for  locally homogeneous Riemannian manifolds.
Mess indeed successfully reduced the geometrical question "{\it What are the globally hyperbolic Cauchy-maximal Cauchy-compact $\mass{}$-manifolds}" to the more algebraic one "{\it What are the representations of $\pi_1(\Sigma)$ into $\isom(\mass{})$ which are  discrete and faithful}".
More precisely,
let $\Sigma$ be a closed surface of genus $g$ and $M$ be  a Cauchy-maximal globally hyperbolic $\mass{}$-manifold admitting a Cauchy-surface homeomorphic to $\Sigma$: 
\begin{itemize}
\item if $g=0$, such an $M$ does not exists;
\item if $g=1$, then 
\begin{itemize}
 \item either $M \simeq \Gamma \backslash \mass{} $  with $\Gamma$ a spacelike translation group; 
 \item or $M \simeq \Gamma \backslash \Omega $   with $\Omega$ the future of a spacelike line $\Delta$ and $\Gamma$ generated by an hyperbolic isometry and a spacelike translation both fixing $\Delta$ setwise.`
\end{itemize}

\item if $g\geq 2 $, then $M\simeq \rho\backslash\Omega$ where $\Omega$ is a convex $\rho$-invariant domain of $\mass{}$  whose support planes are all lightlike and $\rho$ is a  representation $\pi_1(\Sigma)\rightarrow \isom(\mass{})$ whose linear part 
is discrete and faithful. Furthermore, for any such representation there exists a unique such $\Omega$.
\end{itemize}  
 As a consequence, the deformation space of globally hyperbolic Cauchy-maximal $\mass{}$-manifolds admitting a Cauchy-surface homeomorphic to $\Sigma$ of negative genus is naturally identified with the tangent bundle of the Teichmüller space of $\Sigma$ \cite{MR762512}. 
 
On a side note, Mess results circulated as a draft for many years and piqued the interest of the 2+1 quantum gravity community in the 90's \cite{carlip_1998}.  Since then, many works have focused on quantizations of Teichmüller-like spaces 
\cite{MR3480556,MR3709650,MR3950651,MR2233852,moon2019rogeryang} and their interpretation as quantizations for quantum gravity \cite{carlip_1998,MR3523528}, relations with the cosmological microwave  background \cite{planck2019} have been investigated \cite{MR3245885} and more recently link with averaged Einstein equations in relativistic inhomogeneous cosmology \cite{Buchert_2007} has been made \cite{brunswic2020gaussbonnetchern}.

\subsubsection*{Singular flat spacetimes}
In the following, a singular $\mass{}$-manifold, is a second countable Hausdorff topological space $M$ endowed with an {\it almost everywhere} $(\isom_0(\mass{}),\mass{})$-structure. 
By "almost everywhere" we mean the $(G,X)$-structure is only defined on an open subset $\U\subset M$ dense and {\it locally connected in $M$} (recall that for a topological space $X$ and $\U\subset X$, the subset $\U$ is  locally connected in $X$ if for all connected open $\V\subset M$, the intersection $\V\cap \U$ is connected) see \cite{thesis,GXramcover} for a detailed introduction to almost everywhere $(G,X)$-structures. 
Among elementary properties presented in \cite{GXramcover}, such a manifold admits a unique maximal such open $\U$, the {\it regular locus}, denoted $\reg(M)$ the complement of which, the {\it singular locus}, is denoted by $\sing(M)$; furthermore, the time orientation of $\mass{}$ induces a time orientation on $\U$. An almost everywhere (a.e.) $\mass{}$-morphism between two singular $\mass{}$-manifolds $M, N$ is a continuous map $\varphi: M\rightarrow N$ which  (co)restriction $\varphi_{|\U}^{|\V}$ is a $(\isom_0(\mass{}),\mass{})$-morphism for some $\U\subset \reg(M)$ and $V\subset \reg(N)$  dense and locally connected in $M$ and $N$ respectively. Furthermore, if such a $\varphi$ is a local homeomorphism, then $\varphi(\reg(M)) \subset \reg(N)$ and $\varphi_{|\reg(M)}^{|\reg{N}}$ is a $\mass{}$-morphism.

Barbot, Bonsante and Schlenker \cite{Particules_1} carried out a systematic analysis of generalizations of conical singularities to Minkowski space, following an inductive construction suggested by Thurston \cite{MR1668340};  they thus provide plentiful of non trivial  singularity models among which the so-called massive particles and extreme BTZ white holes. The name BTZ comes from Ba\~nados, Teitelboim and Zanelli \cite{Ba_ados_1992} example of 3-dimensional black holes: Barbot, Bonsante and Schlenker have isolated a family of singularities which generalizes  the example of Ba\~nados, Teitelboim and Zanelli and the model of extreme BTZ white holes is  a limit case of this family.
\begin{itemize}
\item Let $\alpha$ be a positive real, the model space of massive particle of mass $2\pi-\alpha$,  denoted by  $\mass{\alpha}$, is the singular $\mass{}$-manifold defined by the flat Lorentzian metric $\d s_\alpha ^2 = -\d t^2+\d r^2+\left(\frac{\alpha}{2\pi} r\right)^2 \d \theta^2$ on $\RR^3$ where $(t,r,\theta)$ are cylindrical coordinates of $\RR^3$;
\item the model space of extreme BTZ white-hole, denoted by  $\mass{0}$, is the singular $\mass{}$-manifold defined by the flat Lorentzian metric $\d s_0 ^2 = -2\d \tau\d r +\d r^2+r^2 \d \theta^2$ on $\RR^3$ where $(\tau,r,\theta)$ are cylindrical coordinates of $\RR^3$. 
\end{itemize}
Notice the Lorentzian metric defining $\mass{\alpha}$ for $\alpha\geq 0$ has vanishing Ricci curvature, it is thus flat. One can check that they are oriented and time oriented (the future direction still being increasing $t$ or $\tau$) and thus normal charts provide natural $\mass{}$-atlases on the complement of the symmetry axis $\{r=0\}$. 
The metric is {\it a priori} singular at $r=0$ and one can also check that the holonomy of the $\mass{}$-structure on $\{r>0\}$ is elliptic of angle $\alpha$ for $\alpha>0$ and parabolic for $\alpha=0$; therefore, $\sing(\mass{\alpha}) = \{r=0\}$ for $\alpha\neq 2\pi$. One can extend the causal and chronological order to the whole $\mass{\alpha}$. A point $p$ in a singular $\mass{}$-manifold is of type $\alpha$ is there exists a neighborhood of $p$ which embeds by an a.e. $\mass{}$-morphism into $\mass{\alpha}$. 
We will show that such a point $p$ has at most one type.
Then for  $A\subset \RR_+$, we define a $\mass{A}$-manifolds as a singular $\mass{}$-manifold $M$ such that for every $p\in \sing(M)$ has type in $A$; futhermore, for $B\subset \RR_+$ we define $\reg_B(M):=\{x\in M \text{ of type }\beta\in B \}$. On such a manifold, one can define both causal and chronological relations, causality and global hyperbolicity makes sense, one can still construct spacelike Cauchy-surfaces \cite{MR2887877,MR3783554} and Cauchy-compactness/completeness  still makes sense. Also, we will show that a chart around a singularity is always a local diffeomorphism, in particular the singular locus is a 1-dimensional submanifold; each connected component is then called a {\it singular line}.
 
\subsubsection*{Main result and generalizations of Mess Theorem}

Mess Theorem has been generalized in many ways during the last 20 years. 
\begin{itemize}
\item  Bonsante \cite{MR2170277,MR2499272} extended Mess Theorem to higher dimension. Among other thing, he reduced the geometrical problem of classifying globally hyperbolic Cauchy-maximal Cauchy-compact $\mathbb E^{1,N}$-manifolds to affine representations of the fundamental group of a hyperbolic manifold into the affine isomorphisms group of $N+1$ dimensional Minkowski space; 
\item by replacing the Cauchy-compactness by Cauchy-completeness,  Barbot \cite{barbot_globally_2004} describes more general links between globally hyperbolic Cauchy-maximal Cauchy-complete $\mathbb{E}^{1,N}$-manifolds and representations of $\pi_1(\Sigma)$ into the group $\isom_0(\mathbb E^{1,N})$ as a consequence he is able again to reduce in a satisfactory way the geometrical classification of such manifolds to representations of the fundamental group of a $N$-manifold into the affine isomorphisms group of $\mathbb{E}^{1,N}$;

\item  Bonsante and Seppi \cite{MR3493421} obtained a result for closed $\Sigma$
 and $\mass{]0,2\pi]}$-mani\-folds with $s$ singular lines of given angles under the
additional hypothesis that $\Sigma$ could be chosen convex; noting $\Sigma^*$ the surface $\Sigma$ with $s$ punctures, they showed one can parameterize globally
hyperbolic Cauchy-maximal $\mass{]0,2\pi]}$-manifolds homeomorphic to $\Sigma\times \RR$ admitting exactly $s$ massive particles of given angles $\beta_1,\cdots,\beta_s$
by the tangent bundle of Teichmüller space of $\Sigma^*$.  
\end{itemize}  
We prove another generalization closely related to the result of Bonsante and Seppi.
\begin{theo}\label{theo:princ_intro} Let $\Sigma^*$ be a surface of genus $g$ with exactly $s$ marked points such that $2-2g-s<0$. The deformation space of globally hyperbolic Cauchy-maximal $\mass{0}$-manifolds homeomorphic to $\Sigma\times \RR$ with exactly $s$ marked singular lines can be identified to the tangent bundle of Teichmüller space of $\Sigma^*$ .
\end{theo}
This result can be associated to Meusberger and Scarinci work on quantization of 2+1 gravity \cite{MR3523528}, indeed their results apply to the regular part of Cauchy-compact $\mass{0}$-manifolds and thus should translate simply to the manifolds considered in the present work.

One should not be surprised by the above result.
Indeed, let $\Sigma^*$  be a surface of genus $g$ with $s$ marked points, by a result of Troyanov and Hulin \cite{MR1005085,MR1166122}, the Teichmüller space of $\Sigma^*$ can be naturally identified with the deformation space of hyperbolic metrics on $\Sigma$ admitting exactly $s$ conical singularities of given angles $\beta_1,\cdots,\beta_s \in [0,2\pi]$; being understood that $\beta_i=0$ means the corresponding point is a cusp. 
 Extreme BTZ singularities are the Lorentzian equivalent of hyperbolic cusps while massive particles are the Lorentzian equivalent of hyperbolic conical singularities: \begin{itemize}
 \item the holonomy of an extreme BTZ singularity is parabolic, the same as a cusp, while the holonomy of a massive particle is elliptic, the same as a conical singularity;
 \item Let $\alpha>0$ (resp. $\alpha=0)$, denote by $\H_\alpha$ the singular hyperbolic plane with exactly one conical singularity of angle $\alpha$ (resp. exactly one cusp) namely: 
  $(\mathbb D^2,\d s_ \alpha^2)$ with $$\d s_\alpha^2 = \left\{ \begin{matrix}4 \frac{\d r^2+\sinh(r)^2 \d \theta^2}{(1-r^2)^2}& \text{if } \alpha>0 \\ \frac{\d r^2+r^2\d \theta^2}{r^2\log(r)^2} 
  & \text{if } \alpha=0 \end{matrix}\right. .$$ One checks that 
 $(\RR_ +^*\times \reg(\H_\alpha)  , - \d T^2+ T^2 \d s_\alpha^2 ) $ is isomorphic to $\reg(I^+(p))$ for some $p\in \sing(\mass{\alpha})$ where $I^+(p):=\{q\in \mass{\alpha} ~|~ q\gg p\}$.
 
\end{itemize}

\subsubsection*{Strategy of the proof}

The general strategy is not original, namely we consider the map $\mathrm{Hol} \circ \reg$ which associates to a  globally hyperbolic Cauchy-maximal and Cauchy-com\-plete $\mass{0}$-manifold the holonomy of its regular part and we construct an inverse to this map. Since the scope of most elements of proof  are  restricted neither to $\mass{0}$-manifolds nor to surfaces of finite type, without lengthning of proofs and as long as deformation spaces are not involved, we work in the setting of $\mass{\geq0}$-manifolds and general metrizable surfaces.

We give ourselves $\Sigma$ a metrizable surface and define the surface $\Sigma^*$ complement of a subset $S$ of at most countable cardinal $s$. 
Consider $M$ a globally hyperbolic Cauchy-maximal and Cauchy-compact $\mass{0}$-manifold homeomorphic to $\Sigma\times \RR$. 

The first step, see section \ref{sec:reg_BTZ}, is to prove the regular part of $M$ is globally hyperbolic Cauchy-maximal and Cauchy-complete.  On the one hand, we prove in section \ref{sec:surgery} technical Lemmas regarding neighborhood of BTZ-lines which allow to construct complete Cauchy-surfaces of the regular part.  On the other hand, we study causal and gluing properties of BTZ-extensions in section \ref{sec:gluing} to pave the way toward absolute maximality of the $\reg(M)$. These technical results are combined in section \ref{sec:Abs_max}  to prove $\reg(M)$ is Cauchy-complete and Cauchy-maximal. 

This allows us to proceed in section \ref{sec:holonomy} with our second step under the additional assumptions that $\Sigma$ is closed of genus $g$, $S$ finite and $2-2g-s<0$: using Barbot's analysis of Cauchy-complete $\mass{}$-manifold as well as classical results on Teichmüller space, we describe in section \ref{sec:admissible_rep} the holonomy of $\reg(M)$ and conclude $\mathrm{Hol}\circ \reg$ defines a well defined map from the deformation space of globally hyperbolic $\mass{0}$-manifolds homeomorphic to  $\Sigma\times \RR$ with exactly $s$ marked singular lines, to the tangent bundle of the Teichmüller space  of $\Sigma^*$.
We call the holonomy given by a point of this moduli space  {\it admissible} and, in section \ref{sec:polyedron}, we construct globally hyperbolic Cauchy-compact $\mass{0}$-manifold of arbitrary admissible holonomy via polyedron gluings; Choquet-Bruhat-Geroch Theorem then gives a Cauchy-maximal such  $\mass{0}$-manifold.

The third step, section \ref{sec:extension_BTZ_2}, achieve the proof of the injectivity of the map $\mathrm{Hol} \circ \reg$: we need to prove  two globally hyperbolic Cauchy-compact Cauchy-maximal $\mass{0}$-manifolds are isomorphic if and only if their regular parts are isomorphic. 
To this end, we introduce the notion {\it BTZ-extension}:

\begin{defou}[BTZ-extension, BTZ-embedding]\

Let $M$ be a $\mass{\geq 0}$-mani\-fold. A  BTZ-extension of  $M$ is an embedding of $M\xrightarrow{\iota} N$ where $N$ is a $\mass{\geq 0}$-manifold and $\iota$ an a.e $\mass{}$-morphism and such that the complement of the image of $\iota$ only contains singular points of type $\mass{0}$.

Such a map $\iota$ is called a BTZ-embedding.
\end{defou}      
 We then prove a result sufficient for our purpose.
 
 \begin{theo} Let $M$ be globally hyperbolic $\mass{\geq 0}$-manifolds, there exists a globally hyperbolic BTZ-extension $M\xrightarrow{\iota} N$ which is maximal among such extensions. Moreover, $M\xrightarrow{\iota}N$ is unique up to equivalence. We  call this extension
the maximal BTZ-extension of $M$ and denote it by 
  $\mathrm{Ext}_{BTZ}(M)$.   
\end{theo}
 Theorem \ref{theo:princ_intro} then follows. 

We add a complement about the good properties of the maximal BTZ-extension with respect to Cauchy-completeness. Indeed, along the way to prove that the regular part of a Cauchy-maximal Cauchy-compact $\mass{0}$-manifold is Cauchy-complete, we developped methods that allow to prove a converse statement that complete nicely the framework of BTZ-extensions. 

	\begin{theo} \label{theo:BTZ_Cauchy-completness_intro}
				Let $M$ be a globally hyperbolic $\mass{\geq0}$-manifold. The following are then equivalent:			
				\begin{enumerate}
				\item $\reg_{>0}(M)$ is Cauchy-complete and Cauchy-maximal;
				\item there exists a BTZ-extension of $M$ which is Cauchy-complete and Cau\-chy-maximal;
				\item $\BTZext(M)$ is Cauchy-complete and Cauchy-maxi\-mal.
				\end{enumerate}
				\end{theo}

\subsection{Acknowledgments}
This work has been part of a PhD thesis supervised by Thierry Barbot  at Université d'Avignon et des Pays de Vaucluse and is part of a project that has received funding from the European Research Council (ERC) under the European Union's Horizon 2020 research and innovation programme (grant agreement ERC advanced grant 740021--ARTHUS, PI: Thomas Buchert).

\subsection{Preliminaries}
Before diving into the proof, we first strengthen the basis of the theory of singular $\mass{}$-manifold we consider. Firstly, we describe the isomorphism group of the model spaces we consider and provide essential analyticity statement. Secondly, we review general causality results which are assumed throughout the present work.

\subsubsection{Isometries and Analyticity} \label{sec:isometries}
Let us begin by an almost everywhere version of a classical analyticity Lem\-ma.
Let $(G,X)$ be an analytical structure with $X$ Hausdorff and locally connected.
In the present work, for $M,N$ singular $(G,X)$-manifolds,  an a.e. $(G,X)$-morphism is a continuous map $f:M\rightarrow N$ such that there exists an open dense and locally connected in $M$ (resp. $N$) subset $\U\subset \reg(M)$ (resp. $\V\subset \reg(N)$) such that $f_{|\U}^{|\V}$ is a $(G,X)$-morphism. Recall that an open $\U$ of a topological space $M$ is locally connected in $M$ if for all connected open $\V$, the intersection $\V\cap \U$ is connected.

\begin{lemou} \label{lem:liouville} Let $M \xrightarrow{f,g} N$ be two a.e. $(G,X)$-morphisms with $M$ connected. If $f$ and $g$ agree on a non-empty open subset $\mathcal W$ of $M$, then they agree on the whole $M$.
\end{lemou}
\begin{proof}
Let $\U\subset \reg(M)$ be a dense open subset of $M$ locally connected in $M$  and let $\V\subset \reg(N)$ be a locally connected in $N$ dense open subset such that $f_{|\U}^{|\V},g_{|\U}^{|\V}: \U\rightarrow \V$ are $(G,X)$-morphisms. Since $f_{|\U}^{|\V},g_{|\U}^{|\V}$ agree on $\mathcal W\cap \U$, by analyticity of $(G,X)$-morphism and connectedness of $M $ hence of $\U=\U\cap M$, $f$ and $g$ agree on the whole $\U$. Since $f$ and $g$ are continuous and $\U$ is dense, then $f=g$ on $M$.
\end{proof}

 Notice that, for any $\alpha\geq 0$, the group of isomorphisms of $\mass{\alpha}$ contains the rotations-translations around and along the singular axis of $\mass{\alpha}$; 
if $\alpha=0$, hyperbolic isometries 
$$\fonction{h_\ell}{\mass{0}}{\mass{0}}{(\tau,r,\theta)}{\left(\ell \tau - \frac{\ell^2-1}{2\ell}r,\ \frac{r}{\ell},\  \ell\theta\right)},\quad\quad \ell \in \N^*$$
also acts on $\mass{0}$. However, such isometries are injective iff $\ell=1$ and $h_\ell$ is the identity. 

We define for $\alpha\geq 0$:
$$\isom(\E^{1,2}_\alpha) = \{(t,r,\theta)\mapsto (t+t_0,r,\theta+\theta_0): (t_0,\theta_0) \in \RR^2\}.$$

These definitions are justified by the next Proposition.

  \begin{propou} \label{prop:isom} Let $\alpha,\beta\in \R_+$  with $\alpha\neq 2\pi$, and let $\U$ be an open connected
	  subset of $\E^{1,2}_\alpha$ containing a singular point
	  and let $\phi:\U\rightarrow \E^{1,2}_\beta$ be an  almost everywhere $\mass{}$-morphism. 
	  
	 If $\phi$ is a local homeomorphism then $\alpha=\beta$ and  
	  $\phi$ is the restriction of an element of $\isom(\E^{1,2}_\alpha)$. 
 \end{propou}
\begin{corou}
For $\alpha\in \RR_+$, $\isom(\mass{\alpha})$ is indeed the group of isometries of $\mass{\alpha}$.
\end{corou}
\begin{corou}
	Let $M$ be a $\mass{\geq 0}$-manifold and let $x\in \sing(M)$. There exists exactly one $\alpha\geq 0$ such that $x$ is locally modeled on $\mass{\alpha}$.
\end{corou}
\begin{corou} Let $\phi:M\rightarrow N$ an a.e. $\mass{}$-morphism with $M,N$ two $\mass{\geq0}$-manifolds. 

If $\phi$ is a local homeomorphism, then $\phi$ is a local diffeomorphisms.
\end{corou}
\begin{corou}A $\mass{\geq 0}$-manifold $M$ is smooth $3$-manifold and its singular locus is a closed 1-submanifold. Furthermore, for any $\alpha\in \RR_+$, $\sing_\alpha(M)$ is a closed 1-submanifold.

\end{corou}

Before proving Proposition \ref{prop:isom} we introduce two singular manifolds that simplify the argumentation. Define the $\mass{}$-manifold \begin{eqnarray*}
	\mass{\infty}&:=& \left((\RR\times \RR_+\times \RR)/\sim, -\d t^2+\d r^2+r^2\d\theta^2\right) \\ && \text{with} \quad  (t,r,\theta)\sim (t',r',\theta') \text{ iff } r=r'=0 \text{ and } t=t'
\end{eqnarray*}

which comes with its natural projection $\varpi_\alpha$ on each $\mass{\alpha}$ with $\alpha>0$. 
For $\alpha=2\pi$, $\mass{\alpha}=\mass{}$ and $\D_\infty:=\varpi_{2\pi}$ is an a.e. $\mass{}$-morphism. 
We can obviously extend our definitions of $\isom(\mass{\alpha})$ to $\alpha = \infty$.
The situation can be summed up by the following commutative diagram (where $\iota$'s are the inclusions):
$$\xymatrix{
\reg(\mass{\infty})~\ar@{^{(}->}[r]^{~~~~\iota_\infty}\ar@{->>}[d]^{\varpi_\alpha}&\mass{\infty} \ar[r]^{\D_\infty}\ar@{->>}[d]^{\varpi_\alpha}& \mass{} \\ 
\reg(\mass{\alpha})~\ar@{^{(}->}[r]^{~~~~\iota_\alpha} & \mass{\alpha} 
}
$$
The manifold $\reg(\mass{\infty})$ is simply connected and, for $\alpha\neq 2\pi$, the map $\reg(\varpi_\alpha)$ is a covering, $\reg(\mass{\infty})$ is thus the universal covering of $\reg(\mass{\alpha})$;
the map $\D_\infty \circ \iota_\infty$ being a $\mass{}$-morphism, it is then the developing map of $\reg(\mass{\alpha})$. Futhermore, if $\D:\mass{\infty}\rightarrow \mass{}$ is another a.e. $\mass{}$-morphism  whose restriction to $\reg(\mass{\infty})$ is a $\mass{}$-morphism,  then there exists $g\in \isom(\mass{})$ such that $g \circ\D\circ \iota_\infty =  \D_\infty\circ \iota_\infty$. The image of $\iota_\infty$ is dense thus $g\circ \D = \D_\infty$ and we recover the usual uniqueness statement for developing maps.  
To summarize, the singular manifold $\mass{\infty}$ is the "universal" branched covering of $\mass{\alpha}$ for $\alpha>0$; using  the terminology of \cite{MR0123298}, $\mass{\infty}$ is the completion of the spread $\reg(\mass{\infty})\xrightarrow{\iota_\alpha\circ \varpi_\alpha}\mass{\alpha}$ and the developing map extends continuously to this spread. 

We define similarly the "universal" covering of $\mass{0}$ branched on $\sing(\mass{0})$:
\begin{eqnarray*}
\mass{0\infty}&:=& \left((\RR\times \RR_+\times \RR)/\sim, -\d \tau^2+\d \r^2+\r^2\d\theta^2\right)  \\&& \text{with} \quad  (\tau,\r,\theta)\sim (\tau',\r',\theta') \text{ iff }  \r=\r'=0 \text{ and } \tau=\tau'
\end{eqnarray*}
together with a natural projection $\varpi_0: \mass{0\infty} \rightarrow \mass{0}$.  
Again one has the commutative diagram: \\
\begin{minipage}[t]{0.405\textwidth}
$$\xymatrix{
\reg(\mass{0\infty})~\ar@{^{(}->}[r]^{~~~~\iota_\infty}\ar@{->>}[d]^{\varpi_0}&\mass{0\infty} \ar[r]^{\D_0}\ar@{->>}[d]^{\varpi_0}& \mass{} \\ 
\reg(\mass{0})~\ar@{^{(}->}[r]^{~~~~\iota_0} & \mass{0} 
} 
$$
\end{minipage}
\begin{minipage}[t]{0.59\textwidth}
$$
\text{where} ~~
 \fonction{\D_0}{\mass{0\infty} }{\mass{}}
{\begin{pmatrix}\tau \\ \r \\ \theta \end{pmatrix}}{
\begin{pmatrix} \tau+\frac{1}{2}\r\theta^2   \\ \tau +\frac{1}{2}\r\theta^2-\r \\ -\r\theta \end{pmatrix}
}
$$
\end{minipage}
We note that 
$$\D_0\left(\mass{0\infty}\right) = J^+(\Delta) \quad  \text{where} \quad \Delta = \RR\begin{pmatrix}1 \\ 1 \\ 0 \end{pmatrix}.$$
The same way as before
$\reg(\mass{0\infty})$ is the universal cover of $\reg(\mass{0})$, the map $\D_0\circ \iota_\infty$ is the developing map of $\reg(\mass{0})$ and furthermore, any other developing map $\D:\Reg(\mass{0\infty})\rightarrow \mass{}$ extends continuously to $\mass{0\infty}$ and there exists $g\in \isom(\mass{})$ such that $g\circ \D = \D_0$. The group $\isom(\mass{0\infty})$ is defined as the group generated by rotation-translations  around and  along $\{r=0\}$ and by hyperbolic isometries $h_\ell$  with $\ell \in \RR_+^*$.

	  \begin{proof}[Proof of Proposition \ref{prop:isom}] 	  
	  To begin with, $\phi$ is a local homeomorphism and an a.e $\mass{}$-morphism thus sends $\sing(\U)\neq \emptyset$ into $\sing(\mass{\beta})$; then $\phi$ certainly sends some simple loop $\gamma$ around $\sing(\mass{\alpha})$ in $\U$ to some loop $\phi(\gamma)$ in $\mass{\beta}$. If $\alpha>0$ (resp. $\alpha=0$), the holonomy of $\gamma$ is elliptic (resp. parabolic), then so is the holonomy of $\phi(\gamma)$, thus  $\beta>0$ (resp. $\beta=0$).
 We prove the Proposition for $\alpha>0$, the proof will work the same way for $\alpha=0$ {\it mutatis mutandis}.
 	  	  
	  One can assume without loss of generality that $\phi$ is an homeomorphism on its image and that $\U$ has the form  $\U = \{t \in ]-\varepsilon,\varepsilon[,\ r \leq r_0\} \subset \mass{\alpha}$. Since $\varpi_\alpha^{-1}(\U)$ is simply connected, the map $\phi$ lifts to a map $\widetilde \phi: \varpi_\alpha^{-1}(\U) \rightarrow \mass{\infty}$ which 
	  sends the singular and regular part of $\varpi_\alpha^{-1}$ into the singular and regular part of $\mass{\infty}$ respectively. Since the continuous maps $\D_\infty\circ \widetilde \phi$ and $\D_{\infty|\varpi_\alpha^{-1}(\U)}$ 
	  (restricted to the regular part of their domains are both developing maps $\reg(\U)$ and since the regular part of $\varpi_\alpha^{-1}(\U)$ is dense in $\varpi_\alpha^{-1}(\U)$, there exists $g\in \isom(\mass{})$ such that $g\circ \D_\infty\circ \widetilde \phi = \D_{\infty|\varpi_\alpha^{-1}(\U)}$. 

Let $\Delta = \D_\infty(\sing(\mass{\infty}))$, consider 	$p,q\in \sing(\mass{\infty})$ distinct and such that $\varpi_\alpha(p)$ and
$\varpi_\alpha(q)$ are  in $\sing(\U)$.  Since 
$p,q\in \sing(\mass{\infty})$ then $\D_\infty(p),\D_\infty(q)\in \Delta$
 and $\widetilde\phi(p),\widetilde\phi(q) \in \sing(\mass{\infty})$ and then $\D_\infty(\widetilde\phi(p))$ and 
$\D_\infty(\widetilde\phi(q))$ are also both in $\Delta$. Furthermore, $g\D_\infty(\widetilde\phi(p)) = \D_\infty(p)$ and $g\D_\infty(\widetilde\phi(q)) = \D_\infty(q)$ thus $g$ sends two distinct points of $\Delta$ into $\Delta$; the isometry $g$ is affine and $\Delta$ is a line thus $g\Delta=\Delta$.

Since the direction of $\Delta$ is timelike, $g$ is an elliptic isometry of axis $\Delta$ with translation part in the direction of $\Delta$; therefore there exists $\widetilde g  \in\isom(\mass{\infty})$ such that $g\circ \D_\infty = \D_\infty \circ \widetilde g$.
We then have $\D_\infty\circ(\widetilde g \circ \widetilde \phi) = {\D_\infty}_{|\omega_{\alpha}^{-1}(\U)}$  and $\widetilde g\widetilde \phi = \widetilde h_{|\varpi_\alpha^{-1}(\U)}$ where $\widetilde h(t,r,\theta) = (t,r,\theta+2k\pi))$ for some $k\in \ZZ$. 
Finally, $\widetilde \phi$  is the restriction of an element of $\isom(\mass{\infty})$.
	For such an element of $\isom(\mass{\infty})$ to induce  a map $\mass{\alpha}\supset\U\rightarrow \mass{\beta}$, then $\alpha$ must be a integral multiple of $\beta$; furthermore for the induced map to be a local homeomorphism one must have $\alpha=\beta$.

\end{proof}

Proposition \ref{prop:isom} can be refined to obtain the stronger following useful Proposition.
\begin{propou}\label{prop:isom_strong}
Let $\alpha,\beta\in \R_+$  with $\alpha\neq 2\pi$, and let $\U$ be an open connected
	  subset of $\E^{1,2}_\alpha$ containing a singular point
	  and let $\phi:\reg(\U)\rightarrow \E^{1,2}_\beta$ be a $\mass{}$-morphism. 
	  
	 If $\phi$ is a injective then $\alpha=\beta$ and  
	  $\phi$ is the restriction of an element of $\isom(\E^{1,2}_\alpha)$. 
\end{propou}

\begin{proof} Using the same argumentation as in Proposition \ref{prop:isom}, we show that $\alpha=0$ (resp. $\alpha>0$) implies $\beta=0$ (resp. $\beta>0$). One only need to prove the result for each $\U$ in a basis of neighborhoods of a singular point of $\mass{\alpha}$. Subsets of the form $J^-(p)\cap J^+(q)$ form a basis of the topology of $\mass{\alpha}$, they are all globally hyperbolic and causally convex in $\mass{\alpha}$. 

From Proposition \ref{prop:isom} and by Theorem of invariance of domain, proving $\phi$ extends continuously and injectively to $\U$ is sufficient. 
Notice that $\phi$ is an injective local homeomorphism thus a $\mass{}$-isomorphism on its image.
Let  $$\fonction{\widetilde \phi}{\U}{\mass{\beta}}{x}{\inf \phi(\Reg(I^+_\U(x)))},$$
\item Claim 1: $\widetilde \phi$ is well defined and $\widetilde \phi_{|\Reg(\U)} = \phi$. 

Let $x\in \U$, $y \in \Reg(I^+_\U(x))$ and let $c: [0,1]\rightarrow \U$ be a finite length timelike geodesic from $y$ to $x$ in $\Reg(\U)$ (which exists since $\U$ is globally hyperbolic), its image $\phi\circ c$ is a finite length timelike geodesic in $\Reg(\mass{\alpha})$. The infimum is realized by $\lim_{1^-}\phi\circ c$ and if $x\in \Reg(\U)$ then $\lim_{1^-} c = x\in \Reg(\U)$ so  $\widetilde \phi(x)=\lim_{1^-}\phi\circ c = \phi(x)$.

\item  Claim 2: $\displaystyle \widetilde \phi(x) = \lim_{n\rightarrow +\infty} \phi(x_n)$ for any $\ll$-decreasing sequence $x_n\xrightarrow{n\rightarrow +\infty}x\in \U$.

Let $(x_n)_{n\in\N}$ be a $\ll$-decreasing sequence converging toward some $x\in \U$.  For all $x'$ in $\phi(\Reg(I^+_\U(x)))$, the past set $I^-(\phi^{-1}(x'))$ is a neighborhood of $x$ and $\exists N\in \N$,  $\forall n\in \N$, $\phi^{-1}(x')\gg x_n$; therefore,  $\exists N\in \N,\forall n\geq N,x'\geq \phi(x_n)\geq \widetilde \phi(x)$. Taking a sequence $(x'_n)_{n\in\N}$ converging toward the infimum of $\phi(\Reg(I^+_\U(x)))$ we obtain $\widetilde \phi(x) = \lim_{n\rightarrow +\infty} \phi(x_n)$.

\item Claim 3: $\widetilde \phi$ sends singular points to singular points.

Let $x\in \sing(\U)$ and  let $x_n\xrightarrow x$ be a $\ll$ decreasing sequence  in $\Reg(\U)$, for every $n\in \N$ let  $\gamma_n$ be a closed loop of non trivial holonomy in $I^-(x_n)\cap I^+(x)$. Then, for every $n\in\N$, $\phi(\gamma_n) \in I^-(\phi(x_n))\cap I^+(\widetilde \phi(x))$. Since $\lim_{n\rightarrow +\infty}\phi(x_n)=\widetilde \phi(x)$, we deduce that in any neighborhood of  $\widetilde \phi$ there exists a closed loop of non trivial holonomy. Therefore, $\widetilde \phi(x)\in\sing(\mass{\beta})$.

\item Claim 4: $\widetilde \phi$ is increasing and injective.

 Since $x\mapsto I^+(x)$ is increasing, $\widetilde \phi$ is non-decreasing. 

Let $x,y\in \U$ such that $\widetilde \phi(x)=\widetilde \phi(y)$. If  $x\in \Reg(\U)$, then $y\in \Reg(\U)$ and $\phi(x)=\widetilde \phi(x)=\widetilde \phi(y)=\phi(y)$ so $x=y$ by injectivity of $\phi$. 
If $x\in \sing(\U)$ then $y\in \sing(\U)$ and then  either $x\leq y$ of $y\geq x$. Assume without loss of generality that $x\leq y$.  

By contradiction, assume $x<y$. Since $\widetilde \phi$ is non-decreasing, for all $z\in \sing(\U)$ with $x\leq z\leq y$, $\widetilde \phi(x)=\widetilde\phi(z)=\widetilde \phi(y)$. One can thus choose a smaller $\U$ such that $\widetilde \phi$ is constant on $ \sing(\U)$. 
Consider $p\in I^+_\U(x)$, then $I^-_\U(p)$ is a neighborhood of $x$ and thus contains some singular $z>x$. Consider two past time-like geodesics $c_x,c_z:[0,1]\rightarrow \mass{\alpha}$ with $c_x(0)=c_z(0)=p$, $c_x(1)=x$ and $c_z(1)=z$. By causal convexity of $\U$, these geodesics lie in $\U$.  Notice these two geodesics only intersect at $p$. Consider their image $\widetilde \phi\circ c_x$ and $\widetilde \phi \circ c_z$; both are geodesics of $\mass{\beta}$ intersecting initially, at $\widetilde \phi(p)$ and finally at $\widetilde\phi(x)=\widetilde\phi(z)$. Notice that the geodesics $\mass{\beta}$ ending on $\sing(\mass{\beta})$ are exactly the radial rays in the cylindrical coordinates $(t,r,\theta)$ or $(\tau,\r,\theta)$ (depending on whether $\beta>0$ or $\beta=0$). Thus, $\widetilde \phi \circ c_x = \widetilde \phi \circ c_z $, hence $c_x\cap \reg(\U)=c_z\cap\reg(\U)$ by injectivity of $\widetilde\phi$ on $\reg(\U)$. Finally, $x=z$; contradiction.

\item Claim 5: {$\widetilde \phi$ is continuous.}

 Since $\phi$ is continuous, is suffices to prove $\widetilde \phi$ is continuous on $\sing(\U)$. Let $p\in \sing(\U)$, consider a $\ll$-decreasing sequence $x_n\xrightarrow{n\rightarrow +\infty}x$ and a $\leq$-increasing sequence $y_n\xrightarrow{n\rightarrow +\infty} x$ such that for all $n\in \N$, $x\in \mathrm{Int}(J^+(y_n))$; if $\alpha>0$ take $y_n\in I^-(p) \neq 0$, if $\alpha=0$ any $y_n\in J^-(p)$ has the wanted property. We already proved that $\phi(x_n)~\xrightarrow{n\rightarrow+\infty}~\widetilde\phi(x)$;  proving $\phi(y_n)\xrightarrow{n\rightarrow +\infty}\widetilde\phi(x)$ will then be enough since $\widetilde \phi$ is increasing.

Now consider some $q\in I^+_\U(p)$, the unique geodesic $c$ from $q$ to $p$ and the sequence of past timelike geodesics $c_n$ from $q$ to $y_n$ for $n\in\N$. For all $n\in \N$, $\phi\circ c_n: [0,1] \rightarrow \mass{\beta}$ is a geodesic and since $(\phi \circ c_n)'(0) \xrightarrow{n\rightarrow +\infty} (\phi \circ c)'(0)$ the sequence $(\phi \circ c_n)_{n\in\N}$ converges uniformly toward $\phi \circ c$. In particular: 
$$\lim_{n\rightarrow +\infty} \widetilde \phi(y_n)=\lim_{n\rightarrow +\infty}\lim_{t\rightarrow 1^-} \phi\circ c_n(t) = \lim_{t\rightarrow 1^-}\lim_{n\rightarrow +\infty} \phi\circ c_n(t)= \widetilde \phi(p).$$


Finally, $\widetilde \phi$ is an injective and continuous extension of $\phi$ to $\U$. It is thus an a.e. $\mass{}$-morphism and a local homeomorphism from $\U$ to $\mass{\beta}$,  Proposition \ref{prop:isom} applies. 
 
\end{proof}

\subsubsection*{Causality of $\mass{\geq0}$-manifolds}

The notion of causal and chronological orders on $\mass{\alpha}$ gives rise to a sheaf with value in the dual category of doubly preordered sets:  
let $\U\subset \mass{\alpha}$ be an open subset of $\mass{\alpha}$, for all $x,y\in \U$,  $x\leq_\U y$ (resp. $x \ll _\U y$) iff there exists a future causal (resp. chronological) curve from $x$ to $y$ in $\U$. Such a sheaf is a {\it causal structure}. 
Clearly, on any $\mass{\geq0}$-manifold $M$, there exists a unique causal structure on $M$ which induces on each chart the causal structure of $\mass{\alpha}$. 
As in \cite{oneil}, the causal future of a point $p\in M$ is $ J^+(p):= \{x\in M ~|~ x\geq p\}$; the causal past $J^-(p)$, chronological past/future $I^\pm(p)$ are defined in similar ways. As for  smooth Lorentzian manifolds, the chronological past and future are always open (possibly empty) subset of $M$.  
Both (strong) causality and global hyperbolicity thus make sense in a $\mass{\geq 0}$-manifold. 

The causal structure around singularities differs slightly from the one of Minkowski space.
 Firstly, with $p$ a singular point of $\mass{0}$, the chronological past of $p$ is empty $I^-(p)=\emptyset$ and its causal past is exactly the singular half-line below $p$: $J^-(p) = \{\r=0,\tau\leq \tau(p)\}$.  This generalizes to any $\mass{\geq0}$-manifolds as follows.
	 \begin{lemou} \label{lem:past_BTZ}
	    Let $M$ a $\mass{\geq0}$-manifold then
	    \begin{itemize}
	    \item a connected component of $\sing(M)$ is an inextendible causal curve;
	    \item every causal curve $c$ of $M$ decomposes into its BTZ part and its non BTZ part $c=\sing_0(c)\cup \reg_{>0}(c)$. Furthermore, $\sing_0(c)$ and $\reg(c)$ are connected causal curves and $$\forall x\in \Reg_{>0}(c),\quad \sing_0(c) \subset J^-(x).$$
	  \end{itemize}
	\end{lemou}
	\begin{proof}
	  A connected component of $\sing_0(M)$ is a closed, connected, locally causal, 1-dimen\-sional submanifold; therefore, 
	  it is an inextendible causal curve. 
	  
	  Let $c:I\rightarrow M$ be a causal curve, we identify $c$ and $c(I)$. If $\sing_0(c)=\emptyset$ or $\reg_{>0}(c)=\emptyset$, there is nothing to prove; 
	  we thus assume $\sing_0(c)\neq\emptyset$ and $\reg_{>0	}(c)\neq \emptyset$. 
	  
	  Let $t_0\in I$ such that $c(t)\in \sing_0(M)$ and let $J:=\{t\in I~|~t\leq t_0, ~ c([t,t_0])\subset \sing_0(M)\}$. Since $
	  \sing_0(M)$	 is closed, so is $J$ and since $J^-(p)\subset 
	  \sing_0(\mass{0})$, it is also open in $I\,\cap\, ]-\infty,t_0]$. Therefore, $J=I\,\cap\, ]-\infty,t_0]$ and the result follows.

	\end{proof}

Geroch Theorem, Choquet-Bruhat-Geroch are still true for $\mass{\geq 0}$-manifolds as well as the smooth spacelike splitting Theorem; their proof though require some extra work due to the causality around singular lines. More precisely.

\begin{itemize}
\item \underline{Geroch Theorem}: a proof by conformal deformation can be found in \cite{Particules_1} for $\mass{]0,2\pi]}$-manifolds, another proof working for $\mass{\geq0}$-manifolds can be obtained by slightly generalizing the original proof of Geroch. Let $M$ be a time-oriented Lorentzian manifolds, a {\it time function} on $M$ is a map $T: M\mapsto \RR$ which is increasing on future causal curves. A time function $T$ is furthermore {\it Cauchy}  if the restriction of $T_{|c}: c \rightarrow \RR$ to any inextendible causal curve $c$ of $M$ is surjective. Geroch proved his Theorem by considering $x\mapsto\ln\left(\frac{\mu(I^-(x))}{\mu(I^+(x))}\right)$ where $\mu$ is any measure absolutely continuous with respect to the Lebesgue measure.  In \cite{thesis}, the author consider $x\mapsto\ln\left(\frac{\mu^-(J^-(x))}{\mu^+(J^+(x))}\right)$ where $\mu^+$ is absolutely continuous with respect to Lebesgue measure and $\mu^- = \mu^+ + \alpha $ with $\alpha $ absolutely continuous with respect to the 1D Lebesgue measure on the BTZ lines of $M$.

\item \underline{Choquet-Bruhat-Geroch}: The proof of Choquet-Bruhat-Geroch Theorem is based upon two  key ingredients: first, the local existence and uniqueness of solutions to Einstein's equation \cite{MR0053338,MR2527641}; second, a causal analysis of the boundary of a globally hyperbolic domain of a globally hyperbolic Lorentzian manifold.  

  Proposition \ref{prop:isom} gives the geometrical equivalent to  local uniqueness.  Then, one can follow the lines of the causal analysis done by Sbierski \cite{sbierski_geroch} taking some care around singularities especially those of type $\alpha>2\pi$ and $\alpha=0$. Key arguments of Sbierski fails {\it as is}  around such singularities but one can easily correct them \cite{thesis}.  
\item \underline{Smooth spacelike splitting Theorem}: First conjectured by Seifert \cite{Seifert1977}, the first accepted proof for Lorentzian manifolds is given by Bernal and Sanchez \cite{sanchez_smooth}. The proof of Bernal and Sanchez does not apply to $\mass{\geq 0}$-manifolds but the recent work of Bernard and Surh \cite{MR3783554} does. 
They consider manifolds endowed with a {\it convex cone field}, in the context of a $\mass{\geq0}$-manifold $M$, it is the field that $\mathcal C$ associate to a point $p\in M$ the future cone in the tangent space of $M$ at $p$. Fathi and Siconolfi \cite{MR2887877} proved that if $\mathcal C$ is continuous, global hyperbolicity implies the existence of a smooth Cauchy time-function on $M$. However, one can check that the future cone field is \textbf{not} continuous around any singular line in a $\mass{\geq0}$-manifold $M$. Bernard and Surh relaxed the continuity hypothesis  and replaced it by assuming the cone field to be closed ie the cone above a point $p$ is the union of the limits of the cone field toward $p$. 
This property is true for $\mass{\geq0}$-manifolds.
 
\end{itemize}

The singularities we are considering also modify some know causal properties. For instance, let $M$ is $\mass{\geq0}$-manifold, $I^-(p)=\emptyset$ for all $p\in \sing_0(M)$  therefore, $M$ is not past distinguishing if it admits at least one extreme BTZ white hole. However, the classical property is true for future distinguishability.  The proof is the same as in \cite{MR0469146}, we reproduce it here to stress the use of causal language as well as the fact the proof rely only on the Causal convexity formulation of strong causality and both facts that $I^{\pm}$ is open and non-empty.

			\begin{propou} \label{prop:FZenon_injectif}Let $M$ be a $\mass{\geq0}$-manifold. If $M$ is strongly causal, then
			the maps $$\fonctionn{M}{\mathcal{P}(M)}{x}{I^+(x)}\quad \quad \fonctionn{M\setminus\sing_0(M)}{\mathcal{P}(M)}{x}{I^-(x)}$$
			are injective.
			\end{propou}
			\begin{proof}
			Let $x,y\in M$ such that $I^+(x)= I^+(y)$. Let $\U_x$ and $\U_y$ two open neighborhoods of $x$ and $y$ respectively. Since $M$ is strongly causal, there exists neighborhoods $\V_x$ and $\V_y$ of $x$ and $y$ respectively, which are causally convex in $M$. Let $a\in \V_x\cap I^+(x)$, since $I^+(x)=I^+(y)$, then $a\in I^+(y)$ and $y\in I^-(a)$. However, $I^-(a)\cap \V_y$ is an open neighborhood of $y$ and thus contains some $b\in I^+(y)\cap I^-(a)\cap \V_y$. 
Again, $I^+(x)=I^+(y)$ so $b\in I^+(x)\cap I^-(a)\subset J^+(x)\cap J^-(a)\subset \V_x$ by causal convexity of $\V_x$. 
Therefore, $b\in \V_x \cap \V_y \subset \U_x\cap \U_y$ and $\U_x\cap \U_y\neq \emptyset$.  
$\U_x$ and $\U_y$ are arbitrary and $M$ is Hausdorff so $x=y$.

The proof works {\it mutatis mutandis} for $I^-$.
			\end{proof}

\section{The regular part of Cauchy-compact $\mass{0}$-manifolds}
\label{sec:reg_BTZ}

 We begin by an easy remark.
 \begin{remark}\label{rem:reg_GH} Let $N$ be a globally hyperbolic $\mass{\geq0}$-manifold, the complement of the BTZ lines $\reg_{>0}(N)$ is globally hyperbolic.
 \end{remark}
 \begin{proof} Let $N^*:=\reg_{>0}(N)$.
 By Geroch Theorem it is enough to check that $N^*$ is causal (which is true since $N$ is causal) and that for any $p,q\in N^*$, the diamond $J^+_{N^*}(p)\cap J^-_{N^*}(q)$ is compact.  
 Let $p,q\in N^*$ and let $r\in J^+_{N}(p)\cap J^-_{N}(q)$ and let $c$ be a future causal curve from $p$ to $q$ containing $r$. If $r\notin N^*$, then $r\in \sing_0(N)$ and by Lemma \ref{lem:past_BTZ}, $q\in \sing_0(N)$ which contradict the hypothesis. 
 Therefore, $J^+_{N^*}(p)\cap J^-_{N^*}(q) = J^+_{N}(p)\cap J^-_{N}(q)$ which is compact since $N$ is globally hyperbolic.
 \end{proof}
 
We are interested in Cauchy-compact Cauchy-maximal $\mass{0}$-manifolds and the aim of this section in to show Proposition \ref{prop:abs_max} ie that under these additional assumptions, $\reg(M)$ is also Cauchy-complete and absolutely maximal in the sense of Barbot \cite{barbot_globally_2004}.

Section \ref{sec:surgery} below is devoted to the proof of Lemma \ref{lem:curve_extension} which, together with technical results regarding BTZ-extensions obtained in section \ref{sec:gluing}, allows us to prove the wanted description of the complement of the BTZ lines of a Cauchy-compact Cauchy-maximal $\mass{0}$-manifold.
\subsection{Cauchy-surfaces surgery in spear neighborhoods}\label{sec:surgery}
    
   Considering the model space $\mass{0}$, it is easy to relate Cauchy surfaces of $\mass{0}$ to Cauchy-surfaces of its regular part $\reg(\mass{0})$. Indeed, consider a Cauchy-surface $\Sigma$ of $\mass{0}$ and let $\T:=\{r\leq \r_0\}$ be the tube of radius $\r_0$ around the singular line of $\mass{0}$. To construct a Cauchy-surface of $\reg(\mass{0})$ from $\Sigma$, it suffices to cut $\Sigma$ along the boundary of $\T$ and replace the piece we took out by a piece of cusp going toward future infinity near the singular line. Converserly, starting from a Cauchy-surface $\Sigma^*$ of $\reg(\mass{0})$, one can construct a Cauchy-surface of $\mass{0}$ by cutting $\Sigma^*$ along $\T$ and replacing the interior piece by a spacelike disk that cut the BTZ line. 
   
   This technical section develop the tools and criterion to ensure the above procedure is feasible, that it can be done in a way preserving Cauchy-complete\-ness and that we keep control on the intersection with causal curves. The tube $\T$ of the example does not exists in general since a BTZ-line has no reason to be complete in the past \footnote{Corollary \ref{cor:abs_max} below shows it's never the case for Cauchy-compact $\mass{0}$-manifold whose fundamental group is anabelian if it has at least one BTZ-line. Indeed, for a Cauchy-complete Cauchy-maximal $\mass{}$-manifold $M$ with anabelian fundamental group, the image of the developing map is a convex of Minkowski space bounded by infinitely many lightlike planes, in particular, it admits a spacelike support plane. Hence, it may contain either a complete future geodesic ray or a complete past geodesic ray  but never both.}. This remark leads to the introduction of spear neighborhoods.
   
   \begin{defou}[Spear neighborhood] Let $M$ be a $\mass{\geq0}$-manifold and let $\Delta$ be a BTZ-line.

   A spear neighborhood around $\Delta$ of vertex $p\in \Delta$ and radius $R>0$ is a neighborhood of $]p,+\infty[$ such that there exists a singular $\mass{}$-isomorphism 
   $\phi: \U\rightarrow \V\subset \mass{0}$ with 
   $$\phi(p)=(-2R,0,0),\quad  \V:= J^+(\phi(p)) \cap \{\r\leq R\}\subset \mass{0}$$ 
   The boundary of $\U$ is the union of the light cone $\phi^{-1}(\{\r = 2(\tau-R),\r<R\})$ and a half-tube $\phi^{-1}(\{\r=R, \tau\geq 0\})$. The former is called the \textbf{head} and the latter the \textbf{shaft} of $\U$.

  \end{defou}
  \begin{defou}[Blunt spear neighborhood]
  Let $M^*$ be a $\mass{>0}$-manifold.
  
   A blunt spear neighborhood of radius $R>0$ in $M^*$ is a subset $\U$ whose interior is dense and locally connected in $\U$ such that there exists a singular $\mass{}$-isomorphism 
   $\phi: \U^*\rightarrow \V^*\subset \mass{0}$ with 
   $$ \V^*:= \reg(J^+(p) \cap \{\r\leq R\})\subset \reg(\mass{0}) \quad\text{ for some } p \in \sing(\mass{0}).$$ 
   
   Such a $\V^*$ is the regular part of a unique spear $\V$ in $\mass{0}$, the \textbf{head} and the \text{shaft} of $\V^*$ are respectively the head and the shaft of $\V$.
  \end{defou}
 Notice that such neighborhoods may not exists, for instance if $M$ is a past half  $\{\tau<0\}$  of $\mass{0}$.  Lemma \ref{lem:tube} shows Global hyperbolicity and Cauchy-maximality are sufficient for a BTZ line  to admit a spear neighborhood.
 The construction of a blunt spear neighborhood require more work and unecessary for our main purpose. We thus postpone the construction of blunt neighborhoods to when it will be actually useful ie in the proof of Proposition \ref{prop:reg_complete}.

    \begin{lemou} \label{lem:tube}
	  Let $M$  be a globally hyperbolic $\mass{\geq0}$-manifold.
	 If  $M$ is Cauchy-maximal then, for any  BTZ line $\Delta$ and any $p\in \Delta$, there exists a spear neighborhood around $\Delta$ of vertex $p$.
	\end{lemou}

      \begin{proof}[Proof of Lemma \ref{lem:tube}]
      Let $\Delta$ be a BTZ-line of $M$ and $p\in \Delta$ arbitrary; let $T$ be a time function whose level sets are $C^1$ submanifolds. Let $\Sigma:=T^{-1}(T(p))$ the level set of $T$ through $p$. 
     
      There exists a neighborhood $\overline \U$ of $p$ isomorphic via some a.e. isomorphism $\phi:\overline \U\rightarrow \mass{0}$ to 
      $$\{\tau \in [\tau_1,R/2,], \r \leq R \} \subset \mass{0} $$
      for some positive real numbers $R,\tau_1>0$ and such that $\tau\circ \phi(p)=0$. Since $M$ is globally hyperbolic, in particular $M$ is strongly causal and one can take this neighborhood small enough so that the surface $\phi^{-1}(\{\tau=R/2,\r<R\})$ is acausal in $M$.
      Without loss of generality, since $\Sigma$ is acausal and transverse to $\Delta$, we can assume $R>0$ small enough so that $\phi(\Sigma\cap \overline \U)$ is a graph of some function $\tau_\Sigma: \mathbb D_R \rightarrow ]-\tau_1,R/2[$. This way, denoting $\Pi=\{\tau=R/2,\r<R\}$  the disk bounding $\phi(\overline \U)$, we have: 
      \begin{itemize}
       \item $\Sigma\subset M_0:=M\setminus J^+\left(\phi^{-1}(\Pi)\right) $;
       
       \item the spear $\S$ of $\mass{0}$ of radius $R$ of vertex $\phi(p)$ is such that its head is in $\overline \U$ and the bottom boundary of its shaft is $\Pi$;
       
        \item the gluing $M_2=\left(M_0\coprod \S \right)/\sim$ is a $\mass{\geq0}$-manifold with  $x\sim y$ if $x\in \U, y\in \S$ and $\phi(x)=y$.
      \end{itemize}

    
     Consider  a past inextendible causal curve in $M_2$ say $c$,
     it can be decomposed into two parts: $c_0=c\cap M_0$ and  $c_1=c\setminus M_0$. The boundary of $M_0$ in $M_2$ is exactly $\Pi$, since $\Pi$ can only be crossed by a past causal curve from the  $\T\setminus M_0$ side to the $M_0$ side, both $c_0$ and $c_1$ are connected.   
     Notice that $\r\circ c_1$ and $\tau\circ c_1$ are both  decreasing, 
     $c_1$ thus has a past end point in $\Pi$ and $c_0$ is then non-empty.     
     $c_1$ does not intersect $\Sigma$ and $c_0$ interests $\Sigma$ exactly once thus $c$ interests $\Sigma$ exactly once.   
     
     Finally, $\Sigma$ is a Cauchy-surface of $M_2$ and $M_0\rightarrow M_2$ is a Cauchy-embedding, since $M$ is Cauchy-maximal  
       $M_2$ embeds into $M$. The spear $\S$ thus embeds in $M$ around $\Delta$.

      \end{proof}

   The following three Lemmas will be useful tools to construct Cauchy-surfaces  around BTZ lines. Lemmas \ref{lem:add_BTZ_param} and \ref{lem:blunt_spear_neighborhood} show that one can choose the spear neighborhood so that a given spacelike Cauchy surface intersects it along a smooth curve on the shaft. Finally, Lemma \ref{lem:curve_extension} allows to extends such a smooth curve on the boundary of the shaft to a smooth surface in the interior of the spear.

   We denote by $\mathbb D_R$ the Euclidean disc of radius $R>0$ whose center of $\mathbb D_R$ is denoted by $O$ and denote by  $\mathbb D_R^*$ the punctured  Euclidean disc of radius  $R>0$.
     We will identify frequently $\mathbb D_R$ with its embedding $\{\tau=0,r\leq R\}$ in $\mass{0}$.

	\begin{lemou}\label{lem:add_BTZ_param}
        Let $M$ be a globally hyperbolic  $\mass{\geq0}$-manifold, let $\Sigma$ be Cauchy-surface of $M$ and let a BTZ line $\Delta$ of $M$ admitting a spear neighborhood.  
	
        Then,  for all  $p\in \Delta\cap I^-(\Sigma)$ , there exists a spear neighborhood $\phi:\mathcal S \rightarrow \mass{0}$ around $\Delta$ of vertex $p$ and some radius $R>0$ such that 
        $\Sigma\cap \S$ lie in the shaft of $\Delta$ and 
        $\phi(\Sigma\cap \mathcal S)\subset \mass{0}$ is the graph of a smooth positive function $\tau_\Sigma: \mathbb D_R\rightarrow \RR_+^*$. 
    \end{lemou}

\begin{lemou}\label{lem:blunt_spear_neighborhood}
        Let $M$ be a globally hyperbolic  $\mass{>0}$-manifold, let $\Sigma$ be Cauchy-surface of $M$ and let $\phi_0: \mathcal S_0\rightarrow \mass{0}$ be a blunt spear neighborhood in $M$.  
	
	Then, there exists a blunt spear neighborhood $\phi:\mathcal S \rightarrow \Reg(\mass{0})$ such that 
	\begin{itemize}
	\item $\mathcal S\subset \mathcal S_0$, $\phi = \phi_{0|\mathcal S}$ and $\inf(\phi(\S))=\inf(\phi(\S_0))$;
	\item  $\Sigma\cap   \mathcal S$ lie in the shaft of $\mathcal S$ and $\phi(\Sigma\cap \mathcal S)\subset \Reg(\mass{0})$ is the graph of a smooth function $\tau_\Sigma: \mathbb D_R^*\rightarrow \RR_+^*$.  	
	\end{itemize}
     \end{lemou}

   \begin{lemou}\label{lem:curve_extension} Let $\mathcal S$ be a spear of vertex $(-2R,0,0)$ and radius $R$  in $\mass{0}$ and 
    let $\tau_{\Sigma}^R: \partial \mathbb D_R \rightarrow \R_+$ be a smooth positive function. Then 
    \begin{enumerate}[(i)]
     \item there exists a piecewise smooth function $\tau_\Sigma: \mathbb D_R\rightarrow \R_+$ extending $\tau_\Sigma^R$ which graph is spacelike and compact.
     \item there exists a piecewise smooth function $\tau_\Sigma:\mathbb D_R^*\rightarrow \R_+$ extending $\tau_\Sigma^R$ which graph is spacelike and complete.
    \end{enumerate}
    \end{lemou}

    \begin{proof}[Proofs of Lemmas  \ref{lem:add_BTZ_param} and \ref{lem:blunt_spear_neighborhood}]

We only prove Lemma \ref{lem:add_BTZ_param} since the proof of Lemma \ref{lem:blunt_spear_neighborhood} is essentially the same.
       
      Let $\phi:\mathcal S\rightarrow \mass{0}$ be a spear neighborhood of some vertex $q\in \Delta$ and some radius $R$ around $\Delta$. Let $p \in \Delta\cap I^-(\Sigma)$, 
      if $q$  is in the past of $p$, then $p\in \mathcal S$, the spear of $\mass{0}$ 
      of radius $R$ and vertex $\phi(p)$ is included in $\phi(\mathcal S)$ to obtain a spear neighborhood around $\Delta$ of vertex $p$; if $q$ is in the future of $q$, the BTZ segment $[p,q]=J^-(q)\cap J^+(p)$ is compact and thus covered by finitely many open charts $(\phi_i,\U_i)_{i\in\lsem 1,n\rsem}$ with $\phi_i(\U_i)=\{\tau\in ]\alpha_i,\beta_i[,\r < R_i\}\subset \mass{0}$. Take $R' = \min(R,R_1,\cdots,R_n)$, from Proposition \ref{prop:isom}, the change of chart from $\U_i$ to $\U_j$ and from $\U_i$ to $\mathcal S$, wherever defined are restriction of rotation-translation of $\mass{0}$ along the singular axis, one can thus certainly choose the $\phi_i$ and $\phi$ in such a way that the change of charts are the identity of $\mass{0}$. We can thus define $\Phi: \mathcal S \cup \bigcup_{i=1}^n \U_i \rightarrow \mass{0}$ by gluing together the $\phi_i$ and $\phi$. The image of $\Phi$ then contains a shaft $\{r<R',\tau>\alpha\}\subset \mass{0}$ containing $\Phi(p)=(\tau_0,0,0)$ with $\tau_0>\alpha$. The spear neighborhood of vertex $\Phi(p)$ and radius $R'$ in $\mass{0}$ is thus in the image of $\Phi$. Finally, $\Delta$ admit a spear neighborhood of vertex $p$ and some radius $R'$. Without loss of generality we may assume that $R'=R$ and that $\Phi(p)=(-R/2,0,0)$.
      We denote by $(\Phi,\mathcal S')$ the spear neighborhood of vertex $p$ constructed.

      Let $w:=\Sigma \cap \Delta $, since $w>p$ and since $\Phi(\Sigma\cap \S')$ is a spacelike 
      surface, hence transverse to the vertical lines $\{\r=\r_0, \theta=\theta_0\}$ in $\Phi(\S')$ for all $(\r_0,\theta_0)\in \mathbb D_R$, there exists $\U$ a neighborhood of $w$ on which $\inf(\tau\circ \Phi(\U\cap \Sigma))>\tau\circ \Phi(p)$ and $\Phi(\U\cap\Sigma)$ is the graph of some smooth function $\tau_\Sigma : \mathbb D_\varepsilon\rightarrow \RR$.
      Without loss of generality, we can choose $R$ small enough so that  $\tau_\Sigma$ is positive and defined over $\mathbb D_R$. 
      
      Consider future lightlike, hence causal, broken geodesics in $\mass{0}$ of the form 
      $$\fonction
      {c_{\theta_0,\r_0}}{
      \RR_+}
      {\mass{0}}{s}{\left\{\begin{array}{ll}\displaystyle ((s-R)/2,s,\theta_0) & \text{if}~~ s\leq \r_0\\ \displaystyle ((s-R)/2,\r_0,\theta_0) & \text{if}~~s> \r_0  \end{array} \right.}$$
      with $\theta_0 \in \R/2\pi\Z$ and $\r_0\in [0,R]$.
     Each $c_{\theta_0,\r_0}$ is a causal curve of $\Phi(\S')$ and thus intersects $\Phi(\Sigma\cap\S')$ at most once. Furthermore, these curves cover $\Phi(\S')$; therefore, $\Phi(\Sigma\cap\S')$ is exactly the graph of $\tau_\Sigma$.

    \end{proof}

   The proof of Lemma \ref{lem:curve_extension} require some preliminary  analysis. Lemma \ref{lem:surface_causal}  provide an effective caracterisation of spacelike graphs in spear neighborhoods. Then Lemma \ref{lem:surface_complete} provides a satisfactory description of graphs in spear neighborhoods that are metrically complete.

     \begin{lemou} \label{lem:surface_causal} Let $\Sigma$ be a piecewise smooth (possibly with boundary) submanifold of $\mass{0}$ which is the graph of some piecewise smooth map $\tau_\Sigma: D \rightarrow \R$ with $D\subset \{\tau=0\}\subset \mass{0}$.
     Then $\Sigma$ is spacelike iff on each smooth domain of $\tau_\Sigma$  $$\displaystyle 1-2\frac{\partial \tau_\Sigma}{\partial \r}-\left(\frac{1}{\r}\frac{\partial \tau_\Sigma}{\partial \theta}\right)^2 >0.$$

      \end{lemou}
      \begin{proof}      
            On a given smooth domain of $\tau_\Sigma$, writing   $\delta =1-2\frac{\partial \tau_\Sigma}{\partial \r}-\left(\frac{1}{\r}\frac{\partial \tau_\Sigma}{\partial \theta}\right)^2$,  
            a direct computation yields: 
            $$\d s_\Sigma^2 =\delta \d \r^2 + \left(\frac{1}{\r}\frac{\partial\tau_\Sigma }{\partial \theta}\d \r-\r \d \theta \right)^2.$$
      $\Sigma$ is then spacelike iff $\delta>0$

%
      \end{proof}

      \begin{lemou}\label{lem:surface_complete}Let $\Sigma^*$ be the graph of some piecewise smooth map $\tau_\Sigma: \mathbb D_R^*\rightarrow \RR$ with $R>0$; then:
      \begin{enumerate}

      \item $\Sigma^*$ is spacelike and complete  if there exists $C>0$ such that
	$$ 1-2\frac{\partial \tau_\Sigma}{\partial \r}-\left(\frac{1}{\r}\frac{\partial \tau_\Sigma}{\partial \theta}\right)^2\geq\frac{C^2}{\r^2}$$ 
	wherever this expression is well defined.
      \item  If $\Sigma$ is spacelike and complete then, 
		$$ \lim_{(\r,\theta) \rightarrow 0}\tau_\Sigma(\r,\theta) = +\infty$$
      \end{enumerate}
      \end{lemou}
      \begin{proof}We use the same notations as in the proof of Lemma \ref{lem:surface_causal}. 
      \begin{enumerate}
       \item

      Let $C>0$  be such as $\delta>\frac{C^2}{\r^2}$. 
      It suffices to prove that a finite length curve in $\Sigma^*$ is extendible.
      Let $\gamma:\R\rightarrow \Sigma$ be  a finite length piecewise smooth curve on 
      $\Sigma$. Write $\gamma(s)=(\tau_\Sigma(\r_\gamma(s),\theta_\gamma(s)),\r_\gamma(s),\theta_\gamma(s))$ for $s\in \R$ and $\ell$ its length.
      Since $$\int_{\R} |\r_\gamma'(s)| \d s\leq \frac{R}{C} \int_{\R} |\r_\gamma'(s)|\frac{C}{\r(s)}\d s \leq \frac{R}{C} \ell <+\infty,$$ 
      then  $\r_\gamma$ converges as $s\rightarrow +\infty$, let
         $\r_\infty:= \lim_{s\rightarrow +\infty} \r_\gamma(s)$. 
      
      For all $a\in \R$, 
      $$\ell \geq \left|\int_0^a C\frac{|\r_\gamma'(s)|}{\r_\gamma(s)}\d s \right|\geq C \left|\ln\left(\frac{\r_\gamma(0)}{\r_\gamma(a)}\right) \right|$$ 
      thus $$\forall a \in \R, ~~\r_\gamma(a)\geq \r_\gamma(0) e^{-\ell/C}>0$$ and thus $r_\infty>0$. 
 
      Take $A >0$ such as $\forall s\geq A, \r_\gamma(s)\in [\r_*,\r^*]$ with $\r_*=\r_\infty/2$ and $\r^*=(\r_\infty+R)/2$
      then for all $b\geq a \geq  A$:
      \begin{eqnarray*}\ell&\geq& \int_{[a,b]} \r_\gamma \left|\frac{\partial \tau_\Sigma}{\partial \theta}\frac{\r_\gamma'}{\r_\gamma^2}-\theta_\gamma' \right| \\
	&\geq&  \int_{[a,b]} \r_*\left(\left|\theta_\gamma' \right|-\left|\frac{\partial \tau_\Sigma}{\partial \theta}\frac{\r_\gamma'}{\r_\gamma^2}\right|\right)\\
	&\geq&  \r_*\int_{[a,b]}
	  \left|\theta_\gamma' \right|
	    -\r_*
	    \left(\sup_{ \mathbb D_{\r^*}\setminus \mathbb D_{\r_*}}\left|\frac{\partial \tau_\Sigma}{\partial \theta}\right|\right) \  
	    \int_{[a,b]}\left|\frac{\r_\gamma'}{\r_\gamma^2}\right|   \\
 \int_{a}^{+\infty}|\theta_\gamma'| &\leq& \frac{1}{\r_*}\left(\ell+\left(\sup_{ \mathbb D_{\r^*}\setminus \mathbb D_{\r_*}}\left|\frac{\partial \tau_\Sigma}{\partial \theta}\right|\right) \frac{\ell}{\r_*} \frac{R}{C} \right)< +\infty
 \\
      \end{eqnarray*}
    so that $\theta_\gamma(s)$ converges as $s\rightarrow +\infty$. 
	Since $\T$ is closed, so is $\Sigma$ and $\tau(\r,\theta)$ thus converges in $\Sigma$. 
	; the curve $\gamma$ is then extendible.
	We conclude that $\Sigma$ is complete.

   \item    Since $\Sigma$ is spacelike,  Lemma \ref{lem:surface_causal} ensures that 
      $1-2\frac{\partial \tau_\Sigma}{\partial \r}-
      \left(\frac{1}{\r} \frac{\partial \tau_\Sigma}{\partial \theta}\right)^2 \geq 0$
      on $\mathbb D_R^*$ wherever well defined. Consider a sequence $(\r_n,\theta_n)\rightarrow 0$, we assume $\r_{n+1}<\frac{1}{2}\r_{n}$, one can construct
      an inextendible piecewise continuously differentiable curve $\gamma=(\tau_\gamma,\r_\gamma,\theta_\gamma):\,]0,R]\xrightarrow{~~} \Sigma$ such that
      \begin{itemize}
       \item  $\forall s\in ]0,R],~\r_\gamma(s)=s$ ;
       \item  $\forall n\in \N,~\theta_\gamma(\r_n)=\theta_n$ ;
       \item  $\forall n\in\N, \forall \r\in ]\r_n,\r_{n+1}[,~ |\theta'_\gamma(\r)|\leq \frac{2\pi}{\r_n} $.
      \end{itemize}		
      Writing $\ell$ the length of $\gamma$, we have:
      $$
      \ell = \int_{0}^{R} \sqrt{1+r^2\theta_\gamma'(r)^2-2 \tau'_\gamma(r)}\d r
       \leq  \int_{0}^{R} \sqrt{5-2 \tau'_\gamma(r)} \d r .
      $$
      The integrand is well defined since  $1+r^2\theta_\gamma'(r)^2-2 \tau_\gamma'(r)>0$. 
      We deduce in particular that $\tau_\gamma'\leq 5/2$ and thus $-\tau_\gamma'\geq |\tau'_\gamma|-5$. By completeness of $\Sigma$, the length $\ell$ of $\gamma$ is infinite thus 
      $\int_{0}^{R}\sqrt{|\tau_\gamma'|}=+\infty$ and thus $\int_0^{R}|\tau'_\gamma|=+\infty$.
      Finally, 
      $$\lim_{r\rightarrow 0}\tau_\gamma (r)=\tau(R)-\int_0^R\left(\tau_\gamma'\right) \geq \int_0^{R}\left(|\tau_\gamma'| -5\right)+\tau(R)=+\infty$$ 
thus    $\lim_{n\rightarrow +\infty}{\tau_\Sigma(\r_n,\theta_n)}=+\infty$.

        \end{enumerate}
      \end{proof}
    \begin{remark}
     For Lemma \ref{lem:curve_extension}, we only need the first part of Lemma \ref{lem:surface_complete} but the second will be useful for proving our constructions give Cauchy-surfaces. 
    \end{remark}

      \begin{proof}[Proof of Lemma \ref{lem:curve_extension}]
      	  \begin{enumerate}[(i)]
	   \item Define $\displaystyle\tau_\Sigma (\r,\theta)= \tau_\Sigma^R(\theta)+M\left(\frac{1}{\r}-\frac{1}{R}\right)$
		 with $\displaystyle M=1+\max_{ \R/2\pi\Z}\left| \frac{\partial \tau_\Sigma^R}{\partial \theta}\right|^2$.
		 \\Then: $ \frac{\partial \tau_\Sigma}{\partial \theta} = \frac{\partial \tau_\Sigma^R}{\partial \theta}$ and 
		 $\frac{\partial \tau_\Sigma}{\partial \r} = -\frac{M}{ \r^2}$. 
		 So that: 
		 $$
		 \delta=1-\left(-\frac{M}{ \r^2} \right) -\frac{1}{\r^2}\left(\frac{\partial \tau_\Sigma^R}{\partial \theta} \right)^2
		 =1+ \frac{M-\left(\frac{\partial \tau_\Sigma^R}{\partial \theta} \right)^2}{\r^2}
		 >\frac{1}{\r^2}
		 $$
		  Therefore, the graph  of $\tau_\Sigma$ is spacelike and complete  by Lemma \ref{lem:surface_complete}.
		  
	  \item Define $$\tau_\Sigma(\r,\theta)=\left\{\begin{array}{ll}
	                                               \left(\frac{2\r-R}{R}\right)^2\tau^R_\Sigma(\theta)+M\left(\frac{1}{\r}-\frac{1}{R}\right)& \text{If } \r\in [R/2,R] \\
	                                               \frac{M}{R}& \text{If } \r\in[0,R/2]
	                                              \end{array}
	                                              \right.$$
	        where $M$ is big enough so that $\delta>0$ for  $\r\in[R/2,R]$.
	       Therefore, the graph  of $\tau_\Sigma$ is spacelike and compact by Lemma \ref{lem:surface_causal}.
	  \end{enumerate}
       
      \end{proof}

    Throughout the present work, given a globally hyperbolic $\mass{\geq0}$-manifold $M$, we will use Lemmas \ref{lem:add_BTZ_param}, \ref{lem:blunt_spear_neighborhood} and \ref{lem:curve_extension} to construct complete Cauchy-surfaces of $\reg_{>0}(M)$ from compact Cauchy-surfaces of $M$, and conversely. The missing ingredient is an efficient way to prove a spacelike surface constructed via Lemma  \ref{lem:curve_extension} is not only spacelike but also a Cauchy-surface. Lemmas \ref{lem:intersect_number_compact} and \ref{lem:intersect_number_complete} give a satisfactory criterion.
    
    \begin{lemou}\label{lem:intersect_number_compact}
      Let $\S := \{\r\leq R\}\cap J^+(p)$ be a spear neighborhood in $\mass{0}$, let $\Sigma$ be a spacelike surface in $\S$  and let $c$ be a future causal curve in $\S$ which is a closed subset of $\S$ and whose restriction to the interior  $\mathrm{Int}(\S)$ of $\S$ is non empty and inextendible. 
      
      Assume $\Sigma$ si the graph of some function $\tau_\Sigma:\mathbb D_R\rightarrow \R$ in the shaft of $\S$.
      Then the cardinal of the intersection $c\cap \Sigma$  depends only on the  position of $\sup(c)\in \S\cup\{+\infty\}$ relative to the circle $\C:=\Sigma\cap \partial \S$: 
      $$|c\cap\Sigma|= \left\{ \begin{array}{ll}
                              1 &  \text{ if } \sup(c) \in J^+(\C)\cup\{+\infty\}\\
                              0& \text{ otherwise }
                             \end{array}
\right. .$$

    \end{lemou}
       \begin{lemou}\label{lem:intersect_number_complete}
      Let $\S^* := \{0<\r\leq R\}\cap J^+(p)$ be a blunt spear neighborhood in $\mass{0}$, let $\Sigma^*$ be a spacelike surface in $\S^*$  and let $c^*$ be a future causal curve in $\S^*$ which is a closed subset of $\S^*$ and whose restriction to $\mathrm{Int}(\S^*)$ the interior of $\S^*$ is non empty and inextendible. 
      
      Assume $\Sigma^*$ is the graph of some function $\tau_{\Sigma^*}:\mathbb D^*_R\rightarrow \R$ in the shaft of $\S^*$ such that $\lim_{(\r,\theta)\rightarrow 0}\tau(\r,\theta)=+\infty$.
      Then the cardinal of the intersection $c^*\cap \Sigma^*$  depends only on the  position of $\sup(c^*)\in \S^*\cup\{+\infty\}$ relative to the circle $\C^*:=\Sigma^*\cap \partial \S^*$: 
      $$|c^*\cap\Sigma^*|= \left\{ \begin{array}{ll}
                              1 &  \text{ if } \sup(c^*) \in J^+(\C^*)\cup\{+\infty\}\\
                              0& \text{ otherwise }
                             \end{array}
\right. .$$

    \end{lemou}
        
    \begin{corou}\label{cor:intersect} Under the hypotheses of Lemmas \ref{lem:intersect_number_compact} and \ref{lem:intersect_number_complete}; if $\reg(c)=c^*$, $\reg(\S)=\S^*$ and $\C=\C^*$
    then $|c\cap \Sigma| = |c^*\cap\Sigma^*|$.
    \end{corou}
 \begin{remark}\label{rem:curve_spear} The hypotheses on the causal curve $c$ in the previous Lemmas may seem tedious but notice that they are satisfied by the connected components of the intersection of an inextendible causal curve of some $\mass{\geq0}$-manifold $M$ with some (blunt) spear neighborhood $\S$ in $M$. 
    \end{remark}
    \begin{proof}[Proof of Lemmas \ref{lem:intersect_number_compact} and \ref{lem:intersect_number_complete}]
     Add a point $+\infty$ at the future causal infinity so that $\S\cup\{+\infty\}$ is diffeomorphic to a closed ball in $\RR^3$ and $\sup(c)=+\infty$ if $\lim\tau\circ c = +\infty$. Four facts are worth noticing: 
     \begin{itemize}
      \item the oriented intersection number of $c$ with $\Sigma$ only depends on the homotopy class with fixed extremities of $c$ \cite{MR2680546};
      \item there is only one homotopy class of curves whose ends are any two given points of $\S\cup \{+\infty\}$; 
      \item $(\S\cup\{+\infty\})\setminus\Sigma$ has exactly two path connected components, the one of $p$ denoted $[p]$ and the one of $+\infty$ denoted $[+\infty]$;
      \item since $c$ is future causal and $\Sigma$ is spacelike, $c$ intersects $\Sigma$  always with positive orientation;
      \item $c$ is a closed subset thus $\sup(c)\in c$.
     \end{itemize}
      Therefore, this intersection number only depends on the connected components of $\inf(c)$ and $\sup(c)$ in $[p]$ or $[+\infty]\cup\Sigma$. Since $c$ is causal and $c\cap \mathrm{Int}(\S)$ is inextendible, on the one hand, $\inf(c)$ is in the head of $\S$ and thus in the connected component of $p$, on the other hand $\sup(c)$ is in the boundary of the shaft of $\S$ or $+\infty$. Therefore, either $\sup(c)\in J^+(\C)\cup\{+\infty\}$ in which case $c$ intersects $\Sigma$ exactly once or $\sup(c)\in J^-(\C)$ and  $c$ does not intersect $\Sigma$.
      
      The proof of Lemma \ref{lem:intersect_number_complete} is similar with two differences.
      \begin{itemize}
       \item The blunt spear is not simply connected. This is easily corrected noticing  that, with given end points, the intersection number does not depends on the number of turns around the BTZ line. Again, the intersection number only depends on the fixed end points.
       \item The assumption $\lim_{(\r,\theta)\rightarrow 0}\tau(\r,\theta)=+\infty$ ensures the past end point of the causal curve  is in the connected component of $[p]$ since  the infimum of $c$  could be on the singular line.
      \end{itemize}

    \end{proof}

\subsection{Gluing and causal properties of BTZ-extensions}
\label{sec:gluing}
We prove in this section several technical results regarding BTZ-extensions that will prove useful. The first Lemma allows to perform easily some gluings.

	\begin{lemou} \label{lem:extension_hausdorff}
	Let $M_0$ be a $\mass{\geq0}$-manifold. 	
					Let $i:M_0\rightarrow M_1$ a BTZ extension and $j:M_0\rightarrow M_2$ an embedding which is an almost everywhere $\mass{}$-morphism  with $M_2$ a $\mass{\geq0}$-manifold.
					Define $M_3:=(M_1\coprod M_2)/M_0$ the pushforward of $i$ and $j$ in the category of topological space, the following diagram commutes:
						$$\xymatrix{
							M_0 \ar[r]^{i}\ar[d]_j & M_1\ar[d]^{\pi_1} \\
							M_2\ar[r]^{\pi_2}& M_3
							}$$
			with $\pi_1$ and $\pi_2$ the natural projections.
			
					If $M_2$ is globally hyperbolic and $\sing_0(M_2)=j(\sing_0(M_0))$, then $M_3$ is a $\mass{\geq0}$-manifold.
	
					\end{lemou}
					\begin{proof}
						Assume $M_2$ globally hyperbolic and $\sing_0(M_2)=j(\sing_0(M_0)$.	It is sufficient to prove $M_3$ is Hausdorff.
						
						Let $p\in M_1$ and $q\in M_2$ such that for all neighborhood 
						$\U$ of $p$ and all neighborhood $\V$ of $q$ :
						$$i^{-1}(\U)\cap j^{-1}(\V)\neq \emptyset.$$
						We shall prove that $\pi_1(p)=\pi_2(q)$.

						Consider a sequence $(a_n)\in M_0^\N$ such that $$   \lim_{n\rightarrow +\infty}i(a_n)=p \quad \text{et}\quad  \lim_{n\rightarrow +\infty}j(a_n)=q .$$
						Since $i$ is a BTZ-extension, for all $x\in M_1, I^+(x)\subset i(M_0)$; and since $M_2$ is globally hyperbolic
						$$ I^+(p) \subset \mathrm{Int}\left(\bigcap_{N\in \N } \bigcup_{n\geq N} I^+(i(a_n))\right) \quad\quad I^+(q) = \mathrm{Int}\left(\bigcap_{N\in \N } \bigcup_{n\geq N} I^+(j(a_n))\right)  .$$

Therefore: 
						\begin{eqnarray*}
						 j\circ i^{-1}(I^+(p))&= & j\circ i^{-1}\left\{ \mathrm{Int}\left(\bigcap_{N\in \N}\bigcup_{n\geq N} I^+(i(a_n))\right)\right\}\\
						 &=& \mathrm{Int}\left(\bigcap_{N\in \N}\bigcup_{n\geq N}j(I^+(a_n)\right)\\
						 &\subset& \mathrm{Int}\left(\bigcap_{N\in \N}\bigcup_{n\geq N} I^+(j(a_n)\right)\\
						 &=& I^+(q).
						\end{eqnarray*}
			
						Consider a decreasing sequence $(b_n)_{n\in \N}\in M_0^\N$ such that $i(b_n) \xrightarrow{n\rightarrow +\infty} p$.
						The sequence $\left(j(b_n)\right)_{n\in \N}$ is decreasing with values in $J^+(q)$; since $M_2$ is globally hyperbolic, $(j(b_n))_{n\in \N}$ converges toward a limit, say $q'\in J^+(q)$.
						
						Let $\U\xrightarrow{\phi} \U'\subset \mass{\alpha}$ be a chart neighborhood of $q'$ and let $\V\xrightarrow{\psi}\V'\subset \mass{0}$ a chart neighborhood of $p$. We assume furthermore that $\U$ is causally convex in $M_2$ and $\V'$ causally convex in $\mass{0}$. 
Let $n\in\N$ such that  $i(b_n)\in \V$, $j(b_n)\in \U$ and $I^+(q')\cap I^-(j(b_n))\subset \U$. Let  $\mathcal W:=I^+_{\V'}(\psi(p))\cap I^-_{\V'}(\psi\circ i(b_n))$, the map:
						$$ \fonction{f}{\mathcal W}{\mass{\alpha}}{x}{\phi\circ j\circ i^{-1}\circ \psi^{-1}(x)}$$
						is well defined, $\mathcal W$ is the regular part of a neighborhood of some point of $\sing(\mass{0})$. The open $\mathcal W$ thus contains a non trivial loop of $\mass{0}$ of parabolic holonomy, its image is thus a non trivial loop of parabolic holonomy in $\mass{\alpha}$. Therefore, $\alpha=0$ and
						$q'\in\sing_0(M_2)=\sing_0(j(M_0))$, thus $q'\in j(M_0)$ and thus
						
						$$i\circ j^{-1}(q')=\lim_{n\rightarrow +\infty} i(b_n) = p.$$ 
						 Finally, $p\in i(M_0)$ so that
						 $$q=\lim_{n\rightarrow +\infty}j(a_n)=\lim_{n\rightarrow +\infty}j\circ i^{-1}\circ i(a_n) = j\circ i^{-1}(p)$$
						 and $\pi_1(p)=\pi_2(q)$.

					\end{proof}

We now prove two causal Lemmas regarding BTZ-extensions. The first gives tools to exploit the following easy fact.
		\begin{remark}  Let $M_0 \xrightarrow{\iota} M_1$ be an embedding of $\mass{\geq0}$-manifold.
		 If $M_1$ is globally hyperbolic and $\iota(M_0)$ is causally convex in $M_1$, then $M_0$ is globally hyperbolic.
		\end{remark}

\begin{lemou}\label{lem:BTZext_caus_conv}
	Let $M_0\xrightarrow{\iota} M_1$ be a BTZ-extension of $\mass{\geq0}$-manifolds. 
	\begin{enumerate}[(a)]
	\item $\iota(M_0)$ is causally convex in $M_1$ if and only if  $$\forall x\in \sing_0(M_0), ~~\sing_0\left[J^+(\iota(x))\right] \subset \iota(M_0).$$ 
	\item If $M_0$ is globally hyperbolic then $\iota(M_0)$ is causally convex in $M_1$.
	
	\end{enumerate}
\end{lemou}
\begin{proof}
Without loss of generality, we assume $M_0\subset M_1$ and $\iota$ is the natural inclusion.
\begin{enumerate}[(a)]
\item Assume $\forall x\in \sing_0(M_0)$, $J^+_{M_1}(x) \cap \sing_0(M_1) \subset M_0$ and consider $p,q\in M_0$ and a future causal curve $c$ of $M_1$ from $p$ to  $q$. By Lemma \ref{lem:past_BTZ}, $c$ decomposes into a BTZ part and a non-BTZ part, the former in the past of the latter. 
If $\sing_0(c)=\emptyset$, then $\sing_0(c)\subset M_0$ trivially. If $\sing_0(c)\neq \emptyset$, then $p\in\sing_0(M_0)$ thus $\sing_0(c)\subset M_0$ by hypothesis. In any case,  $\sing_0(c)\subset M_0$.
Furthermore, $\reg_{>0}(M_1)\subset M_0$, thus $c\subset M_0$.

Assume $M_0$ causally convex in $M_1$ and consider some $p\in \sing_0(M_0)$.  Let $q\in \sing_0(J^+(p))$ and let $c:[0,1]\rightarrow M_1$ be a future causal curve in $M_1$  such that $c(0)=p$ and $c(1)=q$. 
Choose some $q'\in I^+(q)\neq \emptyset$ and extend $c$ to a future causal curve $\widetilde c:[0,2] \rightarrow M_1$ such that $\widetilde c(2)=q'$. Since, $I^+(q)\subset M_0$, both $p$ and $q'$ are in $M_0$ and by causal convexity of $M_0$, $\widetilde c \subset M_0$. In particular, $q=\widetilde c(1)$ is in $M_0$.

\item Assume $M_0$ globally hyperbolic. We consider $p\in \sing_0(M_0)$ and some future causal curve $c:[0,1]\rightarrow M_1$ from $c(0)=p$ to $c(1)\in \sing_0\left[J^+(\iota(x))\right]$. 
Define $I:= \{t_0\in[0,1]~|~ c([0,t_0]) \subset M_0\}$. On the one hand, $M_0$ being open, so is $I$. On the other hand, take any $q'\in I^+(c(t_0))\subset M_0$, the past $J^-_{M_0}(q')$ of $q'$ contains $c([0,t_0[)$. Therefore, $c([0,t_0[)\subset J^-_{M_0}(q')\cap J^+_{M_0}(c(0))$ and by global hyperbolicity of $M_0$, the set $J^-_{M_0}(q')\cap J^+_{M_0}(c(0))$ is compact and  $\lim_{t\rightarrow t_0^-} c(t)\in M_0$. 
The interval $I$ is thus closed. Finally, $I=[0,1]$ and $q\in M_0$.

\end{enumerate}
\end{proof}

This second causal Lemma gives a simple criterion regarding the causality of BTZ-extensions. 
\begin{lemou} Let $M_0\xrightarrow{\iota}M_1$ be a BTZ-extension of $\mass{\geq0}$-manifolds.

If $M_0$ is strongly causal then $M_1$ is causal.
\end{lemou}
\begin{proof}
By contradiction, assume then exists a close future causal curve $c$ in $M_1$. By Lemma \ref{lem:past_BTZ}, either $c\subset \reg_{>0}(M_1)$ or $c\subset \sing_0(M_1)$. 

If $c\subset \reg_{>0}(M_1)$, then $c$ is a closed future causal curve in $\iota(M_0)$ and $M_0$ is thus not causal, hence not strongly causal.

If $c\subset \sing_0(M_1)$, from a finite covering of $c$ by charts and using Lemma \ref{prop:isom}, one can construct a neighborhood  $\U$ of $c$ isomorphic to $\mathcal T:=\langle\gamma\rangle \backslash\{(\tau,\r,\theta)\in \mass{0}~|~ \r<\r^*\}$ for some $\r^*>0$ and $\gamma\in \isom(\mass{0})$ of the form $\gamma: (\tau,\r,\theta)\mapsto (\tau+\tau_0,\r,\theta+\theta_0)$ with $\tau_0>0$. 
In particular the future causal curve $c: t \mapsto (t,\r_1,\theta_1)$ for some $\r_1\in ]0,\r^*[$ and some $\theta_1>0$ is in $\iota(M_0)$ and is either closed (if $\theta_0\in \mathbb Q\pi$) or passes infinitely many times in any neighborhood of any of its points (if $\theta_0\not\in \mathbb Q\pi$); in particular $\iota(M_0)$ is not strongly causal.
\end{proof}
\begin{corou}\label{cor:GH_causal} Any BTZ-extension of a globally hyperbolic $\mass{\geq0}$-manifold is causal.
\end{corou}

\subsection{Absolute maximality}
\label{sec:Abs_max}

Instead of absolute maximality in the sense of Barbot \cite{barbot_globally_2004}, we use a slightly stronger notion of $\mass{A}$-maximality. Let $A\subset \RR_+$ and let $M$ be a $\mass{A}$-manifold. We says that $M$ is $\mass{A}$-maximal  if for every globally hyperbolic $\mass{A}$-manifold $N$ and for every a.e. $\mass{}$-morphism $\varphi:M\rightarrow N$  we have:
 $$ \varphi \text{~is an embedding} \quad \Rightarrow \quad \varphi \text{~is an isomorphism}. $$
 
We first prove the following proposition which is classical in the Lorentzian setting.
\begin{propou}\label{prop:cauchy_abs_max}Let $A\subset \RR_+$ and let $M$ be a globally hyperbolic Cauchy-compact  $\mass{A}$-manifold. The following are equivalent:
\begin{enumerate}[(i)]
\item $M$ is Cauchy-maximal  \item $M$ is $\mass{A}$-maximal.
\end{enumerate}

\end{propou}
\begin{proof} 
If $M$ is $\mass{A}$-maximal, in particular  $M$ is Cauchy-maximal by definition. 
Assume now $M$ is Cauchy-maximal, let $N$ a globally hyperbolic $\mass{A}$-manifold and  $\iota:M\rightarrow N$ be an embedding. We assume  without loss of generality that $M\subset N$ and that $M$ and $N$ are connected and orientable.

Let $\Sigma_M$ be a spacelike Cauchy surface of $M$  and let $N \simeq \Sigma_N\times \RR$ be a smooth spacelike splitting of $N$. Let $\pi$ be the natural projection $\pi:\Sigma_N\times \RR\rightarrow \Sigma_N$ and $T$ the natural projection $T:\Sigma_N\times \RR\rightarrow \R$. Since the splitting is spacelike, the gradient of $T$ is timelike whenever it is well defined ie on $N\setminus \sing(N)$; in particular the fibers above $\Sigma_N\setminus \pi(\sing(N))$ are timelike.

Since $\Sigma_M$ is spacelike, the (causal) fibers of $\pi$ are transverse to $\Sigma_M$ and the restriction of the projection $\phi:=\pi_{|\Sigma_M}:\Sigma_M\rightarrow \Sigma_N$ is a local diffeomorphism; furthermore $\Sigma_M$ is compact and locally compact so $\phi$ is proper; therefore, $\phi$ is a covering. Consider the map $\psi : \Sigma_N \rightarrow \Sigma_M,x \mapsto \inf~(\{x\}\times \RR) \cap \Sigma_M$; it is well defined since $\Sigma_M$ is compact and it is a section of $\phi$. Let $x\in \Sigma_N$, let $\U\subset \Sigma_M $ a neighborhood of $\psi(x)$   and let $\V \subset \Sigma_N$ a neighborhood of $x$ such that $\phi_{|\U}^{|\V}$ is an homeomorphism. 
Since $\Sigma_M$ is a submanifold and $\phi$ a covering, we can choose $\U$ small enough so that there exists $\varepsilon>0$ such that $$\forall x\in \U, \forall y\in \phi^{-1}(\phi(x)),\quad x\neq y \Rightarrow |T(x)-T(y)|\geq \varepsilon$$
Assume by contradiction that for all open $\V'\subset \V$ neighborhood of $x$, there exists a $ x'\in \V'$ such that $\psi_{|\V'}(x') \neq (\phi_{|\U}^{|\V})^{-1}(x')$; there thus exists a sequence $(y_n)_{n\in\N} \in \Sigma_M^\N $ such that $\phi(y_n)\xrightarrow{n\rightarrow +\infty} x$ and $$\forall n\in\N,\quad  T(y_n)\leq  T\circ (\phi_{|\U}^{|\V})^{-1}\circ \phi(y_n)-\varepsilon$$ for some $\varepsilon>0$ small enough.  By compactness of $\Sigma_M$ one may assume $(y_n)_{n\in\N}$ converges toward some $y\in \Sigma_M$. We have $\phi(y) = x$ and $T(y)\leq T(\psi(x))-\varepsilon$. This contradicts the definition of $\psi$. 
As a consequence, $\psi_{|\V'} = (\phi_{|\U'}^{|\V'})^{-1}$ for some $\V'$ neighborhood of $x$ small  enough and some $\U'$ neighborhood of $\psi(x)$ small enough; we thus deduce that $\psi$ is a continuous section of $\phi$ (in particular injective), a local homeomorphism and, since $\Sigma_M$ and $\Sigma_N$ are compact connected, a covering (in particular surjective) hence an homeomorphism and so is $\phi$.  In particular, $\phi_*:\pi_1(M)\rightarrow \pi_1(N)$ is onto;
from Lemma 45 p427 of \cite{oneil} (which proof applies {\it as is} to $\mass{\geq0}$-manifolds), $\Sigma_M$ is achronal.

Since $\sing(N)$ is a 1-submanifold of $N$, then $\Sigma_N\setminus \pi(\sing(N))$ is non-empty and the fiber of $\pi$ above some $p \in \Sigma_N\setminus \pi(\sing(N))$ is timelike. In particular, there exists an inextendible timelike curve in $N$.

On the one hand, inextendible causal curves of $N$ are all homotopic (with  fixed end points "at future and past infinity"); on the other hand the intersection number of an inextendible timelike curve with $\Sigma_M$ is $1$. Furthermore, a future causal curve always intersects $\Sigma_M$ in the same direction. Therefore, 
 by standard intersection theory results \cite{MR2680546}, every inextendible causal curves intersect $\Sigma_M$ exactly once. 
$\Sigma_M$ is thus a Cauchy-surface.

Finally, $N$ is a Cauchy extension of $M$ and by Cauchy-maximality of $M$, we have $M=N$.

\end{proof}

We now prove the main result of this section.
\begin{propou}\label{prop:abs_max} We give ourselves $A\subset \RR_+$ and $A^*:=A\setminus \{0\}$.

 Let $M$ be a globally hyperbolic Cauchy-compact $\mass{A}$-manifold. 
 If $M$ is Cauchy-maximal then $\reg_{>0}(M)$ is globally hyperbolic $\mass{A^*}$-maximal and Cau\-chy-complete.
\end{propou}
\begin{corou}\label{cor:abs_max}
 Let $M$ be a globally hyperbolic Cauchy-compact $\mass{0}$-mani\-fold. 
 If $M$ is Cauchy-maximal then $\Reg(M)$ is globally hyperbolic Cauchy-complete and absolutely maximal in the sense of Barbot \cite{barbot_globally_2004}.
\end{corou}
 
\begin{proof}[Proof of Proposition \ref{prop:abs_max}]
 Assume $M$ is Cauchy-maximal and denote $M\setminus \sing_0(M)$ by $M^*$.
 Let $\Sigma$ be a Cauchy-surface of $M$, for each BTZ line $\Delta$ of $M$, one can apply successively Lemma \ref{lem:tube} to show $\Delta$ admits a spear neighborhood, Lemma \ref{lem:add_BTZ_param} to construct a spear neighborhood $\mathcal S_\Delta$ around $\Delta$ whose vertex in the past of $\Sigma$ and such that $\Sigma\cap \mathcal S_\Delta$ is a  smooth spacelike circle in the shaft of $\mathcal S_\Delta$; and finally the point  $(ii)$ of Lemma \ref{lem:curve_extension} to extend this circle to obtain a Cauchy-surface of $\mathcal S_\Delta \setminus \Delta$.   
 Since $M$ is Cauchy-compact, it admits finitely many BTZ-lines and one can take the spear neighborhoods $\mathcal S_\Delta$ disjoint with $\Delta$ running across the BTZ lines of $M$. We can then construct a surface $\Sigma^*$ equal to $\Sigma$ outside the spear neighborhoods $\mathcal S_\Delta$ and equal to the graph obtained by the point $(ii)$ of Lemma \ref{lem:curve_extension} inside each $\mathcal S_\Delta$. 

  The surface $\Sigma^*$ is metrically complete and, by Remark \ref{rem:curve_spear} and Corollary \ref{cor:intersect} applied to the in each spear neighborhood, a Cauchy-surface of $M^*$. 

  We now prove $M^*$ is $\mass{A^*}$-maximal. Let $N$ be a globally hyperbolic $\mass{A^*}$-manifold and let $\iota$ be an embedding $M^*\xrightarrow{~\iota~} N$.  Since $0\notin A^*$, $\sing_0(M^*)=\sing_0(N)=\emptyset$; then, by Lemma \ref{lem:extension_hausdorff}, the pushforward $\overline M$ of $\iota$ and the natural inclusion $M^*\rightarrow M$ is a $\mass{A}$-manifold.  By Proposition \ref{prop:cauchy_abs_max}, $M$ is $\mass{A}$-maximal so  we only need to prove that $\overline M$ is globally hyperbolic to obtain that the embedding $M\rightarrow \overline M$ is surjective and thus that $\iota$ is surjective and an isomorphism. In the following argumentation, the past and future $J^\pm$ are taken in $\overline M$ if not specified otherwise.

 \begin{itemize}
  \item Notice that $N\rightarrow \overline M$ is a BTZ-extension thus, by Corollary \ref{cor:GH_causal}, $\overline M$ is causal. 
  \item Let $p\leq q$ in $\overline M$. If $p\in N$, in particular $p\notin \sing_0(\overline M)$ and $q\notin \sing_0(\overline M)$; therefore $q\in N$ and, by Lemma \ref{lem:BTZext_caus_conv},
$J^+(p)\cap J^-(q) = J^+_N(p) \cap J^-_N(q)$ which is compact. 
We thus assume $p\in \sing_0(\overline M)=\sing_0(M)$. If $q$ is a BTZ point then, by Lemma \ref{lem:past_BTZ}, $J^+(p)\cap J^-(q) = J^+_M(p) \cap J^-_M(q)$ which is compact.

 We are then left with the case $p\in \sing_0(M)$ and $q\in N$. 
Consider a spear neighborhood $\mathcal S$ in $M$ of vertex $p$ such and let $\mathcal C$ be common boundary of the shaft and the head of this spear neighborhood. The spear $\S$ can be chosen in such a way that $q\notin \S$.

Since no past causal curve can enter $\S$ via its head and $J^+(\C)$ contains the boundary of the shaft of $\S$; then define $K := J^+(\mathcal C)\cap J^-(q)$. Since $\C\subset N$, by causal convexity of $N$, we have
 $$K := J^+(\mathcal C)\cap J^-(q)=J^+_{N}(\mathcal C)\cap J^-_{N}(q).$$ By global hyperbolicity of $N$ and compactness of $\C$, we deduce that $K$ is compact. Define $K':=K\cap \partial S\subset M$, since $\partial S$ is closed, then $K'$ is compact.   Therefore  $J^-_M(K')\cap J^+_M(p)$ is compact by global hyperbolicity of $M$. 
Now, consider a past causal curve $c:[0,1]\rightarrow \overline M$ from $q$ to $p$. Let $t_0 = \max c^{-1}(K')$, since $q\notin \S$, $t_0$ is well defined and positive.
On the one hand, $c([t_0,1])\subset J^-(q)\cap J^+(p)\cap \S\subset M$ thus 
$c([t_0,1]) \subset J^-_M(K')\cap J^+_M(p)$. On the other hand, $c(t_0)\in K'$ so $c([0,t_0])\subset J^+_N(K')\cap J^-_N(q)$ which compact.
Finally,
$$ J^+(p)\cap J^-(q) \subset \left[J^-_N(q)\cap J^+_N(K')\right] ~\cup~  \left[J^-_M( K' )\cap J^+_M(p)\right],$$
however, the reverse inclusion is trivial and
each term of the union is compact; $ J^+(p)\cap J^-(q) $ is thus compact. 
 \end{itemize}

\end{proof}

\section{The holonomy of a globally hyperbolic Cau\-chy-compact Cauchy-maximal $\mass{0}$-manifold}
\label{sec:holonomy}
Let $\Sigma$ be a genus $g$ closed surface, $S\subset \Sigma$ be a marking of $\Sigma$ with $\#S=s$ and $2g-2+s>0$. As before, we write $\Sigma^* = \Sigma \setminus S$. Let $M$ be a globally hyperbolic Cauchy-compact Cauchy-maximal $\mass{0}$-manifold with a Cauchy-surface homeomorphic to $\Sigma$ and with exactly $s$ BTZ lines. These notations will be used throughout the section. Our objective is to obtain a characterise the holonomy of such a manifold $M$.


\subsection{Admissible representations and holonomy}
\label{sec:admissible_rep}

Define $M^*:=\reg(M)$ and let $\rho$ be the holonomy of $M^*$. Let $a_1,\cdots,a_g$, $b_1, \cdots,b_g$, $c_1,\cdots,c_s$ 
be generators of $\pi_1(\Sigma^*)$  such that $(a_i,b_i)_{i\in \lsem 1,g\rsem}$ are interior each associated to a handle of $\Sigma^*$ and $(c_i)_{i\in \lsem 1,s\rsem}$ are peripheral each associated to a puncture. 
We split $\rho$  into its linear part $\rho_L: \pi_1(\Sigma)\rightarrow \GG$ and its translation $\rho$-cocycle $\tau: \pi_1(\Sigma)\rightarrow \mass{0}$.

\begin{defou}[Tangent translation part] Let $\phi$ be an affine parabolic isometry of $\mass{}$, we note $\phi_L$ its linear part and $\tau_\phi$ its translation part.
We say $\tau_\phi$ is tangent if $\tau_\phi$ is normal to the direction of line of fixed points of $\phi_L$.
 
\end{defou}
\begin{defou}[Admissible representation ] \ \\ Let $\Gamma = \left\langle a_1,b_1,\cdots,a_g,b_g, c_1,\cdots,c_s \left |  \prod_{i=1}^g [a_i,b_i] \prod_{j=1}^s c_j=1\right.\right\rangle$
be a marked surface group.
A marked representation $\rho:\Gamma \rightarrow \isom(\E^{1,2})$ is admissible if 
\begin{itemize}
 \item its linear part $\rho_L: \Gamma \rightarrow \GG$ is discrete and faithful;
 \item $\rho_L(c_i)$ is parabolic for all $i\in \lsem 1,s\rsem$;
 \item its translation part $\tau_\rho(c_i)$ is tangent for every $i\in \lsem 1,s\rsem$.
\end{itemize}
\end{defou}

By Proposition \ref{prop:abs_max}, $M^*$ is a globally hyperbolic Cauchy-complete $\mass{}$-maximal $\mass{}$-manifold with anabelian fundamental group; Barbot \cite{barbot_globally_2004} proved that the linear part $\rho_L$ of the holonomy of such a manifold is discrete and faithful. In particular, since in addition $\pi_1(\Sigma)$ is finitely generated, the holonomy of interior generators are hyperbolic \cite{MR1177168} and, again from Barbot \cite{barbot_globally_2004} the holonomy of each of the $s$ peripheral generators is either parabolic or hyperbolic.
 Since $M$ admits exactly $s$ BTZ lines, the holonomy of the peripheral generators is given by the holonomy around the BTZ lines which is parabolic. Furthermore, Barbot also shows (see \cite{barbot_globally_2004}  section 7.3) that if the holonomy of some loop is parabolic, then  the translation part is tangent. 
This discussion can be summarized.

 \begin{propou}  
The holonomy of a globally hyperbolic Cauchy-compact $\mass{0}$-mani\-fold is admissible.
 \end{propou}

  The marked representation $\rho_L : \Gamma\rightarrow \GG$ is thus the holonomy of a unique marked finite volume complete hyperbolic surface with exactly $s$ cusps; from Troyanov and Hulin \cite{MR1166122}, $\rho_L$ is thus a point of the Teichmüller space of $\Sigma^*$.
Denote by $\mathfrak{so}(1,2)$  the Lie algebra of $\mathrm{SO}_0(1,2)$.
From Goldman \cite{MR762512}, we learn that the tangent space  of Teichmüller space above $\rho_L$ is the set of cocycles $\tau\in H^1(\rho_L,\mathfrak{so}(1,2))$ 
such that for each $i\in \lsem 1,s\rsem$, $\tau(x_i)$ is normal to the line of fixed points of $\rho_L(c_i)$ for the Killing bilinear form.
 One notices  that the Killing bilinear form of $\mathfrak{so}(1,2)$ is of signature $(1,2)$ so that $H^1(\rho_L,\mathfrak{so}(1,2))$ can be understood as the set of $\tau:
 \pi_1(\Sigma^*)\rightarrow \mass{}$ such that $\rho_L+\tau$ is an affine representation  of $\pi_1(\Sigma^*)$ (see for instance \cite{MR2499272} section 3.8). Hence, the condition given by Goldman on a $\tau$ in the tangent space above $\rho_L$  is equivalent to the statement that $\tau$ is tangent. 

\begin{propou}[\cite{MR762512,MR1166122}]
 The tangent bundle of the Teichmüller space of $\Sigma^*$ identifies with the set of equivalence classes of admissible representations of $\pi_1(\Sigma^*)$ into $\isom(\mass{})$.
\end{propou} 

We can sum up these two properties as follow.

\begin{propou}  \label{prop:hol_map}
The holonomy of a globally hyperbolic Cauchy-compact $\mass{0}$-manifold is admissible and
 the map 
$$ \M_{g,s}(\mass{0})\xrightarrow{~~\mathrm{Hol}\,\circ\, \Reg~~} T\T_{g,s}$$ is well defined.
\end{propou}

\subsection{Globally hyperbolic $\mass{0}$-manifold of given holonomy}
\label{sec:polyedron}
Let $\Sigma$ be a genus $g$ closed surface, $S\subset \Sigma$ be a marking of $\Sigma$ with $\#S=s$ and $2g-2+s>0$. As before, we write $\Sigma^* = \Sigma \setminus S$ and we assume $s>0$. We denote by $$\Gamma:= \left\langle a_1,b_1,\cdots,a_g,b_g, c_1,\cdots,c_s \left |  \prod_{i=1}^g [a_i,b_i] \prod_{j=1}^s c_j=1\right.\right\rangle$$ a presentation of $\pi_1(\Sigma^*)$.
The previous subsection identified the holonomy of a globally hyperbolic Cauchy-compact $\mass{0}$-manifold homeomorphic to $\Sigma\times \RR$ with $s$ BTZ lines:  it is an admissible representation of $\Gamma$ into the group of affine isometries of Minkowski space.
Starting from such an admissible representation $\rho: \Gamma \rightarrow \isom(\mass{})$, our goal is to construct a globally hyperbolic Cauchy-complete spacetime of given admissible holonomy.  

Denote the linear part of $\rho$ by $\rho_L:\Gamma\rightarrow \GG$, by
   $\mathbb K$ the Klein model of the hyperbolic plane (resp. $\partial \mathbb K$ its boundary) in the projectivisation of $\mass{}$, namely the set future timelike (resp. lightlike) rays from the origin of Minkowski space $\mass{}$ which is parametrized by the futur timelike (resp. lightlike) vectors whose  $t$ coordinate is $1$;  we denote by $\H$ the hyperboloid model of the hyperbolic plane in Minkowski, namely the set of future timelike vectors of norm -1. The group $\GG$ acts on $\H$, $\mathbb K$ and $\partial\mathbb K$ via the usual  matrix multiplication.

 Let $\mathcal T:= \left(T_i\right)_{i\in \lsem 1,n\rsem}$ be an ideal triangulation of $\rho_L\backslash \mathbb K$, and denote by $\widetilde{\mathcal T}:= \left(\widetilde T_{i,\gamma}\right)_{i\in \lsem 1,n\rsem,\gamma\in\Gamma}$ a lift of $\mathcal T$ to an ideal triangulation of $\mathbb K$. In what follows, the ideal triangles contains their vertices.  Let $(u_{i})_{i\in \lsem 1,s\rsem} \in (\partial \mathbb K)^s$ such that $\rho_L(c_i)\in \stab(u_i)$ and such that for all triangle $T_{i} = [u_{i_1}u_{i_2}u_{i_3}]$ the triplet $(u_{i_1},u_{i_2},u_{i_3})$ is a direct base of the vector space underlying $\mass{}$. 
 Define  $\Lambda := \{\rho_L(\gamma) u_i:i\in \lsem1,s\rsem, \gamma \in \Gamma \}\subset \partial\mathbb K$  the infinity set of $\Gamma$.  
 For each $i\in \lsem 1,s\rsem$, define $\Delta_{i}:= \fix(\rho(c_i))$ the line of fixed point of the affine isometry $\rho(c_i)$ in Minkowski space and choose some $p_i \in \Delta_i$.  Since $\rho_L$ is faithful, the $u_i$ are distinct and so are the $\Delta_i$ for $i\in \lsem 1,s\rsem$.

Consider the simplicial complex $\Stilde = \left(\Stilde_{i,\gamma}\right)_{i\in \lsem 1,n\rsem,\gamma\in \Gamma}$ given by the triangulation $\widetilde{\mathcal T}$ of the set $\mathbb K\cup \Lambda$ with $\Stilde_{i,\gamma}$ the simplex associated to $T_{i,\gamma}$ and define $\widetilde{\S}^*$ as the complement of the 0-facets of $\Stilde$. We parameterize each simplex by a standard simplex $T=\{\alpha\in [0,1]^3~|~\sum_{i=1}^3\alpha_i = 1\}$ the quotient $\S := \rho_L\backslash \Stilde$ is homeomorphic to $\Sigma$. A simple way to construct a singular $\mass{}$-manifold with the wanted holonomy is to define 
 a  $\rho$-equivariant local homeomorphism $\mathcal \D:  \R_+^*\times \Stilde^* \rightarrow \mass{}$. The group $\Gamma$  acts on the a.e. $\mass{}$-structure pulled back by $\D$ via a.e. $\mass{}$-morphism and the quotient $M:=\Gamma\backslash(\RR_+^*\times \Stilde)$ is a singular $\mass{}$-manifold. The regular part of $M$ is $\reg(M)=\Gamma\backslash(\RR_+^*\times \Stilde^*)$ and its holonomy is $\rho$.
 As a natural choice for $\D$, fix some $\kappa \in \RR_+$, then for  $ t\in \RR_+^*$ and $\alpha \in \Stilde_{i,\gamma}$ for some $(i,\gamma)\in\lsem1,n\rsem\times \Gamma$, we define 
 \begin{eqnarray*}
   \D(t,\alpha)& := &\rho(\gamma)\left[(\kappa+t)\left\{\alpha_1u_{i_1}+\alpha_2  u_{i_2}+\alpha_3u_{i_3}\right\}
  +\alpha_1 p_{i_1} + \alpha_2p_{i_2}+\alpha_3 p_{i_3}\right]
\end{eqnarray*}
This map is piecewise smooth, its image is the intersection of a half-space with a ruled domain foliated transversally by totally geodesic triangles. For each $i\in \lsem 1,n\rsem$ and  for $t$ big enough, the affine part becomes negligible, hence, for $\kappa$ big enough, on each 3-facet, the map $\D$ has a non singular Jacobian matrix and preserves orientation. By compactness of $\Sigma$, we can choose a uniform $\kappa$, we thus obtain a local homeomorphism on the complement of the 1-facets and the pullback of the $\mass{}$-structure of $\mass{}$ by $\D$ defines a $\mass{}$-structure on $\RR_+^*\times \Stilde^*$ hence on $M := \rho\backslash(\RR_+^*\times \Stilde^*)$.  Moreover, the signature of the induced metric on  the $t=cte$ leaves is given by the signature of the Gram  matrix $$ \begin{pmatrix} \langle e_1|e_1\rangle & \langle e_1|e_2\rangle \\ \langle e_1|e_2\rangle &\langle e_2|e_2\rangle \end{pmatrix}\quad \text{where} \quad \left\{ \begin{matrix} e_1 = (t+\kappa)(u_{i_2}-u_{i_1})+p_{i_2}-p_{i_1} \\  e_2 = (t+\kappa)(u_{i_3}-u_{i_1})+p_{i_3}-p_{i_1}  \end{matrix}\right. . $$
For $\kappa$ big enough  the affine part becomes negligible and the signature becomes $(+,+)$ as for any triplet of future lightlike vectors. So for $\kappa$ big enough, the $t=cte$ leaves are spacelike and the coordinate $t$ is a $\rho$-equivariant time function.  Finally, provided we can prove the singular lines are locally isomorphic to a "reasonable" model space like $\mass{\omega}$ for some $\omega\geq 0$, we can check $M$ is globally hyperbolic by showing  that $t$ is a Cauchy-time function, ie that for any inextendible causal curve $c$, its restriction $T\circ c : \RR \rightarrow \RR_+^*$ is surjective. Indeed, for an inextendible causal curve $c$ assume $t_0=\inf t\circ c>0$; by compactness of $\S$ there exists some $x\in \S$ such that $p=(t_0,x)$ is an accumulation point of $c$.  Consider a chart neighborhood $\U$ of $p$; since the surface $t=t_0$ is a spacelike it is locally acausal and we can choose $\U$ small enough so that it is a Cauchy-surface of $\U$. Then, any inextendible causal curve of $\U$ goes through the spacelike surface $t=t_0$, in particular $c\cap \U$ intersects the surface $t=t_0$ which contradicts the definition of $t_0$. We may proceed the same way to prove the supremum of $t\circ c$ is $+\infty$.

However, we didn't prove the singular points are locally modeled on $\mass{0}$ and the form of the map $\D$ makes this delicate. To avoid this difficulty, we twist the developing map above to have spear neighborhoods around the singular lines; the idea is to force $\widehat \D$ being affine in the neighborhood of the singular lines.

Fix $\kappa\in \RR_+$ and consider for $(i,\gamma) \in \lsem 1,n\rsem\times \Gamma$ and $(t,\alpha,\beta)\in \RR_+^*\times \Stilde_{i,\gamma}\times \Stilde_{i,\gamma}$
 \begin{eqnarray*}
   \P(t,\alpha,\beta)& := &\rho(\gamma)\left[t\left(\alpha_1u_{i_1}+\alpha_2  u_{i_2}+\alpha_3u_{i_3}\right) 
  +\kappa\left(\beta_1u_{i_1}+\beta_2  u_{i_2}+\beta_3u_{i_3}\right)\right.\\&&\quad \quad +\left. \beta_1 p_{i_1} + \beta_2p_{i_2}+\beta_3 p_{i_3}\right].
\end{eqnarray*}
Note that for all $t$ and $\alpha$ in some $\Stilde_{i,\gamma}$ we have  $\D(t,\alpha) = \P(t,\alpha,\alpha)$.
Choose any continuous piecewise smooth  map  $\varphi : T \rightarrow T, \alpha\mapsto (\varphi_1(\alpha),\varphi_2(\alpha),\varphi_3(\alpha))$ such that:
\begin{enumerate}[(i)]
 \item $\varphi_i(\alpha) = 1$  if $\alpha_i\geq 2/3$ for $i\in\{1,2,3\}$;
 \item $\varphi$ is $\mathfrak S_3$-equivariant for the natural action on $T$ by permutation of coordinates (in particular $\varphi(1/3,1/3,1/3)=(1/3,1/3,1/3)$).
 \item The restriction of $\varphi$ to the  open hexagonal domain $H = \{\forall i, \alpha_i< 2/3\}$ is a diffeomorphism  $H \rightarrow T^*$ with $T^*= T\setminus\{\text{vertices}\}$. 
 \item The differential  of $\varphi$ (wherever defined) has a non-negative spectrum.
\end{enumerate}

Then consider the $\rho$-equivariant map
$$\fonction{\widehat{\D}}{\RR_+\times\Stilde}{\mass{}}{t,\alpha\in \Stilde_{i,\gamma}}{ \P(t,\varphi(\alpha),\alpha)}.$$ 
 The second item ensures $\widehat \D$ is well defined on the 2-facets of $\RR_+\times\Stilde$. As before we choose $\kappa$ big enough so that the affine part $\sum_{j=1}^3\alpha_j p_{i_j}$ is negligible. The third and fourth items garantee that the restriction of $\widehat\D$ to any of its smooth subdivision is an orientation preserving embedding, hence that $\widehat\D$ is a local homeomorphism on the complement of the 1-facets of $\RR_+^*\times \Stilde$. The map $\widehat\D$ thus induces a $\mass{}$-structure on $M:=\rho\backslash(\RR_+^*\times \Stilde^*)$. 
  As before, for $\kappa$ big enough, the leaves $t=cte$ are spacelike and the $t$ coordinate is a time function. 
  
  We now focus on proving the singular lines are locally modeled on $\mass{0}$. Consider a peripheral loop $c$, up to reordering the triangulation of $\Sigma$, we may assume the 2-cells of $\Stilde$ around the vertex  associated to $c$ (which we also denote by $c$) are given by the sequence
  $$(\Stilde_{n})_{n\in \ZZ} =   \cdots,\Stilde_{1,c^{-1}},\cdots,\Stilde_{m,c^{-1}} ,\Stilde_{1,1},\cdots,\Stilde_{m,1},\Stilde_{1,c},\cdots,\Stilde_{m,c},\cdots $$
  We call $u_c$ (resp. $p_c$) the lightlike vector (resp. the point) associated to $c$ and for each $n\in \ZZ$ we set $\widetilde T_{n} = [u_c v_{n}v_{n+1}]$ with $\widetilde T_n$ the ideal triangle of $\mathbb K$ corresponding to $\widetilde S_{n}$. The image of $\RR_+\times \{c\}$ is the lightlike line $\Delta$ directed by $u_c$ through $p_c$, denote by $\Pi_n$ the half-plane bounded by $\Delta$ and whose direction contains $v_n$.
  We finally introduce the neighborhood $\U=\bigcup_{n\in\ZZ} \U_n$ of $c\in \Stilde$ with $\U_n = \{\alpha \in \Stilde_n~|~ \alpha_c\geq 2/3 \}$. Each $\U_n$ is an affine triangle bounded by edges $e_n,e_{n+1}$ and $f_n$ where $(e_n)_{n\in \ZZ}$ are the directed edges of $\U$ containing from $c$ numerated so that $(\partial_t,\overrightarrow{e_{n}},\overrightarrow{e_{n+1}})$ is direct. 
  
  For $\kappa$ big enough, we have  $$\forall t\in\RR_+^*,\forall x\in \U, \quad  \langle\widehat\D(t,x)-p_{c}|u_c\rangle\leq0$$ 
  so that the image of $\RR_+^*\times \U$ is in the half-space $J^+(\Delta)$.
  We parameterize $J^+(\Delta)$ by $\mass{0\infty}$ via $p_c+\D_0$ with $\D_0$ defined as in section \ref{sec:isometries}:
   $$\fonction{\D_0}{\mass{0\infty} }{\mass{}}
{\begin{pmatrix}\tau \\ \r \\ \theta \end{pmatrix}}{
\begin{pmatrix} \tau+\frac{1}{2}\r\theta^2   \\ \tau +\frac{1}{2}\r\theta^2-\r \\ -\r\theta \end{pmatrix}.
}$$
  Then, the coordinates $\tau,\r,\theta$ makes sense on $J^+(\Delta)$ and in particular for any point in the image of $\widehat\D_{|\U}$.
   For $n\in\ZZ$, the map $\widehat\D_{|\RR_+\times \U_n}$ is an affine embedding and 
  the plane $\Pi_n$ is the support plane of  $\widehat\D(\RR_+\times e_n)$. 
  Therefore, the image of $\widehat\D_{|\RR_+\times \U_n}$ is the prism adjacent to $\Delta$ bounded by the half-planes $\Pi_n,\Pi_{n+1}$ and the support plane of $\widehat\D(\RR_+\times f_n)$ in the future of the support plane of $\widehat\D(\{0\}\times \U_n)$. Each of the half-planes $\Pi_n$ is a constant $\theta$ surface there thus exists a sequence $(\theta_n)_{n\in\ZZ}$ such that $\forall n\in\ZZ, \Pi_n=\{\theta=\theta_n\}$ and since $\widehat\D$ preserves orientation, $\forall n\in\ZZ,\theta_{n+1}>\theta_n$ and $\widehat \D(\RR_+\times \U_n)\subset \{\theta\in[\theta_n,\theta_{n+1}]\}$; in particular, $\widehat\D_{|\RR_+\times \U}$ is injective. The holonomy $\rho(c)$ acts by fixing $\Delta$ and sending the points $(\tau,\r,\theta)$ of $J^+(\Delta)$ to $(\tau,\r,\theta+\Theta)$ for some $\Theta\in\RR_+^*$, up to composing $\widehat \D$ by a hyperbolic isometry we may assume $\Theta=2\pi$,  moreover $\rho(c)$ sends $\Pi_n$ to $\Pi_{n+r}$ so $\forall n\in\ZZ, \theta_{n+r}-\theta_n =  2\pi$. We deduce that $\lim_{n\rightarrow \pm \infty} \theta_n=\pm\infty$. 
  With $q=\widehat \D(0,c)\in \Delta$, there thus exists a future lightlike vector $v\neq u_c$ such that for all $\varepsilon\in [0,1]$ and all $\tau_0$ the horocycle $q+\tau_0u_c+ \varepsilon\,\stab(u_c)v$ is in $\widehat\D(\RR_+\times \U)$.
  The quotient by $\rho(c)$ of the  domain covered by these horocycles is a spear neighborhood in $J^+(\Delta)/\rho(c)\simeq \mass{0}$.
  
  Finally, the line $\RR_+^*\times c \subset M$ admits a spear neighborhood and is thus  locally modeled on $\mass{0}$. We can conclude following the same arguments as before to show the $t$ coordinate is a Cauchy time function and that $M$ is globally hyperbolic. We thus proved the following Proposition.

\begin{propou}\label{prop:surjective}
  The holonomy map $$ \M_{g,s}(\mass{0}) \xrightarrow{~\reg~\circ~\Hol~} T\T_{g,s}$$
  is surjective.
\end{propou}

  \section{Maximal BTZ extension} \label{sec:extension_BTZ_2}

      We now focus on the last step of the proof of Theorem \ref{theo:princ_intro}. Namely, given two globally hyperbolic Cauchy-maximal Cauchy-compact $\mass{0}$-manifolds say $M$ and $N$, assuming $\reg(M)$ and $\reg(N)$ are isomorphic, we wish to prove that $M$ and $N$  are isomorphic. To this end, we introduce the notion of BTZ-extension.
   
	Note that, since any $\mass{\geq 0}$-manifold is a 3-manifold,  an  a.e $\mass{}$-morphism is an embedding if and only if it is injective.
        
        We prove the following Theorem a corollary of which is the wanted result.
      \setcounter{theo}{1}

      \begin{theo}[Maximal BTZ-extension] Let $M$ be globally hyperbolic $\mass{\geq 0}$-manifolds, there exists a globally hyperbolic BTZ-extension $M \xrightarrow{\iota} \BTZext(M)$ of $M$ which is maximal among such extensions. Moreover, $M\xrightarrow{\iota}  \BTZext(M)$ is unique up to isomorphism.
  
      \end{theo}
       
      In particular, if $M$ is a BTZ-maximal globally hyperbolic $\mass{\geq0}$-manifold then any BTZ-extension $M\xrightarrow{\iota}N$ into a globally hyperbolic manifold $N$ is surjective hence an isomorphism.
   \begin{corou}\label{cor:common_reg} Let $M,N$ be globally hyperbolic Cauchy-maximal Cauchy-com\-pact $\mass{0}$-manifolds: 
   $$\reg(M)\simeq \reg(N) \quad \Rightarrow{} \quad M\simeq N.$$
   \end{corou}
    \begin{proof}
   Since $M$ and $N$ are Cauchy-compact and Cauchy-maximal, by Proposition \ref{prop:cauchy_abs_max}, they are $\mass{\geq0}$-maximal. In particular, $M$ and $N$ are BTZ-maximal. Furthermore, by Remark \ref{rem:reg_GH}, $\reg(M)$ and $\reg(N)$ are globally hyperbolic. Assuming $\reg(M)\simeq \reg(N)$, Theorem \ref{theo:BTZ_ext} thus gives $$M \simeq \BTZext(\reg(M)) \simeq \BTZext(\reg(N)) \simeq N.$$ 
   \end{proof}
      The proof of Theorem \ref{theo:BTZ_ext} has similarities with the one of the Choquet-Bruhat-Geroch Theorem. The proof relies on the existence of the pushfoward
of two BTZ-extensions: 
given two globally hyperbolic BTZ extensions $M_1$ et $M_2$
		of a globally hyperbolic $M_0$, we first a construct a maximal sub-BTZ-extension common $M_1\wedge M_2$ to  $M_0\rightarrow M_1$ and
		 $M_0\rightarrow M_2$ then we glue $M_1$ to $M_2$ along $M_1\wedge M_2$ to obtain $M_1\vee M_2$ the minimal BTZ-extension common to $M_1$ and $M_2$.

		 $$
			
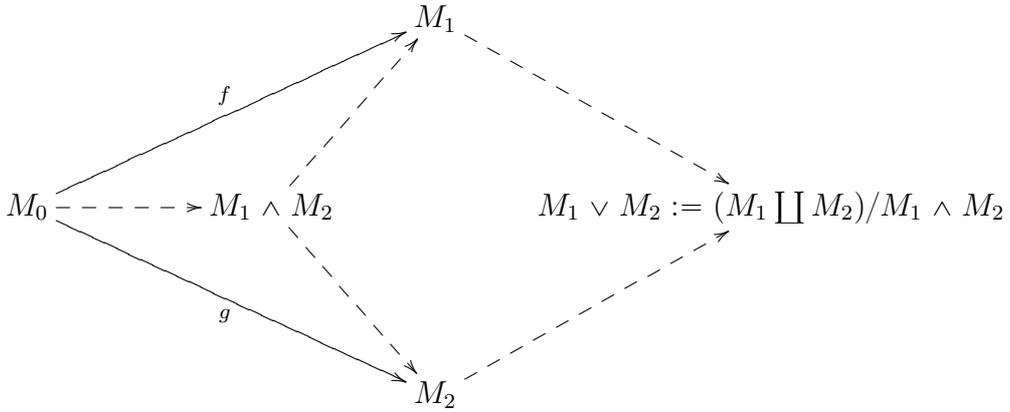
\xymatrix{
				&&& M_1\ar@{-->}[ddr] &\\ &&&&\\
			M_0\ar[ddrrr]_g\ar[uurrr]^f\ar@{-->}[rr]&&M_1\wedge M_2\ar@{-->}[uur]\ar@{-->}[ddr]& &M_1\vee M_2:=(M_1\coprod M_2)/M_1\wedge M_2\\ &&&&\\
			&&&M_2\ar@{-->}[uur]&
			}
		$$
	 
		 The key element is to prove the gluing of $M_1$ and $M_2$ is Hausdorff so that it inherits a natural $\mass{\geq0}$-structure and then to prove this  $\mass{\geq0}$-structure is globally hyperbolic. 
		
		This proves the family of the BTZ-extensions of a given globally hyperbolic $\mass{\geq0}$-manifold $M_0$ is right filtered. The inductive limit of such a family is 
		Hausdorff and naturally endowed with a $\mass{\geq0}$-structure but one still need to check it is second countable and globally hyperbolic.

		\subsection{Maximal common sub-BTZ-extension}
		\begin{defou}[Common sub-BTZ-extension] Let $M_0$ be a $\mass{\geq0}$-mani\-fold, let $M_0\xrightarrow{i} M_1$ and  $M_0 \xrightarrow{j} M_2$
			two BTZ-extensions of $M_0$.
			
			A common sub-BTZ-extension to $M_1$ and $M_2$ is BTZ-extension  $M_0\rightarrow M$  together with two BTZ-embeddings $M\xrightarrow{a} M_1$, $M \rightarrow M_2$ such that the following diagram commutes:
			
			$$
				
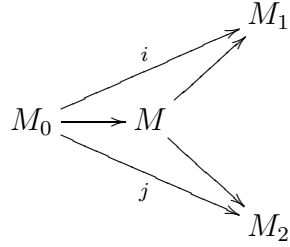
\xymatrix{
					& &M_1\\
					M_0\ar[urr]^i \ar[drr]_j\ar[r]&M\ar[ur]\ar[dr]& \\
					&& M_2
				}		
			$$
		
		\end{defou}
		
		\begin{defou}[Morphism of common sub-extension]
			Let $M_0$ be a $\mass{\geq0}$-manifold, let $M_0\xrightarrow{i} M_1$ and $M_0 \xrightarrow{j} M_2$
			be two BTZ-extensions. Let $M$ and $M'$ two common sub-BTZ-extensions to $M_1$ and $M_2$.
			
			A morphism of sub-BTZ-extension is a $\mass{\geq0}$-morphism $M\xrightarrow{~~\phi~~} M'$such that the following diagram commutes:
			$$
				
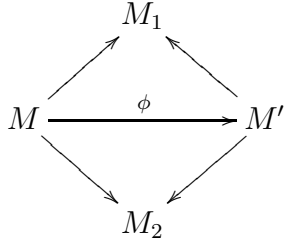
\xymatrix{
					 &M_1 &\\
					M\ar[ur]\ar[dr]\ar[rr]^\phi& &M'\ar[ul]\ar[dl]\\
					& M_2 &
				}		
			$$
			
			If $\phi$ is bijective, then $M$ and $M'$ are equivalent. 	
		\end{defou}
		\begin{defou}[Maximal common sub-BTZ-extension]
				Let $M_0$ be a $\mass{\geq0}$-manifold and let $M_0\xrightarrow{} M_1$ and $M_0 \xrightarrow{} M_2$
			two BTZ-extensions  of $M_0$. Let $M$ be a common sub-BTZ-extension to $M_1$ and $M_2$.
			
			$M$ is maximal if for all common sub-BTZ-extension $M'$ to $M_1$ and $M_2$,  there exists an injective morphism of common sub-BTZ-extension from $M'$ to $M$.
		\end{defou}

		\begin{propou}\label{prop:pgcd} Let $M_0$ be a $\mass{\geq0}$-manifold and let $M_0\xrightarrow{i} M_1$ and $M_0 \xrightarrow{j} M_2$
			two BTZ-extensions  of $M_0$. 

			If $M_0,M_1$ and $M_2$ are globally hyperbolic then there exists a maximal common sub-BTZ-extension to
			$M_1$ and $M_2$.
		\end{propou}

		\begin{proof} Assume $M_0$, $M_1$ and $M_2$ are globally hyperbolic.
		\begin{itemize}
			\item Let $\pgcd$ the union of all the globally hyperbolic open subset $M\subset M_1$ containing $i(M_0)$ such that there exists a BTZ-embedding $M\xrightarrow{\phi_M} M_2$ which restriction to $i(M_0)$ is $j\circ i^{-1}$. The open $M_1\wedge M_2$ is well defined since we can choose $M=i(M_0)$.
				For such an $M$, the map $\phi_M$ is unique since it continuous and equal to $j\circ i^{-1}$ on $\reg(M)$ which is dense. Define $$\fonction{\phi}{\pgcd}{M_2}{x}{\phi_M(x) \mathrm{~if~} x \in M}$$
				Clearly, $i(M_0)\subset \pgcd$ and the inclusions $M_0\xrightarrow{i}\pgcd \rightarrow M_1$ are BTZ-embeddings.
				Since $j: M_0 \rightarrow M_2$ is a BTZ-embedding and $j(M_0)\subset \phi(\pgcd)$  we see that $\phi(\pgcd)$ is dense in $M_2$ and that $M_2\setminus\phi(\pgcd) \subset \sing_0(M_2)$  .
				
				Let $x,y \in \pgcd$ such that $p:=\phi(x)=\phi(y)$.
				Let $M_x$ (resp. $M_y$) an globally hyperbolic open subset of $M_1$ containing $i(M_0)$ and $x$ (resp. $y$).
				Notice that $I^+(p)\subset M_2\setminus\sing_0(M_2)$ thus
				$$\emptyset\neq I^+(p)=\phi(I^+(x))=\phi(I^+(y))\subset j(M_0);$$
				and thus $I^+(x)=I^+(y)$. Since $M_1$ is globally hyperbolic, it is strongly causal hence $x=y$, by Proposition \ref{prop:FZenon_injectif}.
				
				Finally, $\phi$ is a  BTZ-embedding.

			\item 	We now prove that $M_1 \wedge M_2$ is globally hyperbolic. Since $M_1$ is strongly causal, so is $\pgcd$. Let $p, q \in M_1 \wedge M_2$ and let $M_p\subset M_1 \wedge M_2$ be a globally hyperbolic BTZ-extension of $i(M_0)$ in  which contains $p$.
			
			Either $q \in \sing_0(M_1)$ and
					$J^+_{M_1}(p)\cap J^-_{M_1}(q)$ is a possibly empty BTZ line segment, hence compact; or $q\in i(M_0)$ and we have $$J^+_{M_1}(p)\cap J^-_{M_1}(q) \supset J^+_{\pgcd}(p)\cap J^-_{\pgcd}(q)\supset J^+_{M_p}(p)\cap J^-_{M_p}(q),$$
					observing that 
					$J^+_{M_1}(p)\cap J^-_{M_1}(q) = J^+_{M_p}(p)\cap J^-_{M_p}(q)$ which is compact by global hyperbolicity of $M_p$, thus $J^+_{\pgcd}(p)\cap J^-_{\pgcd}(q)$ is compact.

				Then $M_1\wedge M_2$ is globally hyperbolic.

			\item By construction, $\pgcd$ is maximal.
		\end{itemize}
      \end{proof}
\subsection{Minimal common over-BTZ-extension}
		\begin{defou}[Common over-BTZ-extension ] Let $M_0$ be a $\mass{\geq0}$-mani\-fold and let $M_0\xrightarrow{i} M_1$ and $M_0 \xrightarrow{j} M_2$
			two BTZ-extensions  of $M_0$. 

		A common over-BTZ-extension BTZ to $M_1$ and $M_2$ is a  $\mass{\geq0}$-manifold $M$ together with BTZ-embeddings
		$M_1\xrightarrow{a} M$; $M_2 \xrightarrow{b} M$ such that the following diagram commutes:

		$$
			\xymatrix{
				& M_1&\\
				M_0\ar[ur]^i \ar[dr]_j&&M\ar@{<-}[ul]_{a}\ar@{<-}[dl]^{b} \\
				& M_2&
			}		
		$$

		\end{defou}

		\begin{defou}[Morphism of common over-BTZ-extension ]
		Let $M_0$ be a $\mass{\geq0}$-manifold and let $M_0\xrightarrow{i} M_1$ and $M_0 \xrightarrow{j} M_2$
			two BTZ-extensions  of $M_0$.  Let $M$ and $M'$ be two over-BTZ-extensions common to $M_1$ and $M_2$.

		A morphism of over-BTZ-extensions from $M$ to $M'$ is a $\mass{\geq0}$-morphism $M\xrightarrow{~~\phi~~} M'$
		 such that the following diagram commutes:
		$$
			\xymatrix{
					&M_1 &\\
				M\ar@{<-}[ur]\ar@{<-}[dr]\ar[rr]^\phi& &M'\ar@{<-}[ul]\ar@{<-}[dl]\\
				& M_2 &
			}		
		$$

		If $\phi$ is bijective then $M$ and $M'$ are equivalent. 	
		\end{defou}

		\begin{propou}[Pushforward of two BTZ-extensions] \label{prop:ppcm}Let $M_0$ be globally hyperbolic $\mass{\geq0}$-manifold, let $M_0\xrightarrow{i} M_1$ and $M_0 \xrightarrow{j} M_2$
			be two globally hyperbolic BTZ-extensions of $M_0$. There exits a globally hyperbolic over-BTZ-extension $\ppcm$ common to
			$M_1$ et $M_2$	which is minimal among such extensions.

			Furthermore, $\ppcm$ is unique up to equivalence.
	
		\end{propou}
		\begin{proof}
			
			We identify $\pgcd$ as an open subset of $M_1$ so that the natural embedding $\pgcd \rightarrow M_1$ is the natural inclusion. We then denote by $\phi$ the natural embedding $\pgcd \rightarrow M_2$ and define the topological space $\ppcm:= (M_1\coprod M_2)/\sim$ where $x\sim y$ if $\phi(x)=y$ or $\phi(y)=x$. We need to show that $\ppcm$ is Hausdorff and admits a globally hyperbolic $\mass{\geq0}$-structure such that the natural maps $M_1,M_2\xrightarrow{\pi_1,\pi_2} \ppcm$ are $\mass{\geq0}$-morphism. If one shows that the quotient is Hausdorff, the maps $\pi_1,\pi_2$ are then homeomorphism on their image and the $\mass{\geq0}$-structure of $M_1$ and $M_2$ induces two
$\mass{\geq0}$-structure on the image of $\pi_1$ and $\pi_2$. 	
Since $\pi_{1|\pgcd} = \pi_2\circ\phi$, the $\mass{\geq0}$-structures agree on the intersection of the image, namely $\pi_1(\pgcd)$, and define a $\mass{\geq0}$-structure on $\ppcm$. Therefore, one only needs to show that $\ppcm$ is Hausdorff and globally hyperbolic.
			
			 Two points $p,q$ of $\ppcm$ are unseparated if  for all neighborhoods $\U,\V$ of $p,q$ respectively the intersection $\U\cap \V $ is non empty. A point $p$ is unseparated if there exists $q\neq p$ such that $p,q$ are unseparated. 
			Define $C$ the set of  $p\in M_1$ whose image in $\ppcm$ is unseparated. We shall prove that $C$ is empty, to this end, we first show that $\pgcd \cup C$ is connected open and that $\phi$ extends injectively to $C$;  finally, we show that $\pgcd\cup C$ is
			globally hyperbolic; the maximality of $\pgcd$ will then implies that $C\neq \emptyset$.
			\begin{itemize}
				\item 	Notice that $C\subset \sing_0(M_1)\subset \overline \pgcd$.
					Let $p\in C$ and $p' \in M_2$ such that $p$ and $p'$ are unseparated in $\ppcm$. Let
					$\U_p\xrightarrow{\phi_p} \V_p \subset \mass{0}$ be a chart neighborhood of $p$ and $\U_{p'}\xrightarrow{\phi_{p'}} \V_{p'} \subset \mass{\alpha}$
					be a chart neighborhood of $p'$ for some $\alpha \in \RR_+$. Since $p$ and $p'$ are unseparated
					in $\ppcm$, there exists a sequence $(p_n)_{n\in \N} \in \reg(M_1)^\N$ such that
					$\lim_{n\rightarrow +\infty } p_n =p$ and $\lim_{n\rightarrow +\infty} \phi(p_n) = p'$.
					Consider such a sequence and notice that
					$$\forall n\in \N,\quad I^+(p_n)\subset j(M_0) ~\text{ and }~\phi\left(I^+(p_n)\right)=I^+(\phi(p_n)) $$
	so that
					\begin{eqnarray*}
						I^+(p') &=&	Int\left\{\bigcap_{N\in \N} \bigcup _{n\geq N} I^+(\phi(p_n))\right\} \\
							&=&	\phi\left(Int\left\{\bigcap_{N\in \N} \bigcup _{n\geq N} I^+(p_n)\right\}\right) \\
					I^+(p') &=& \phi(I^+(p)).
					\end{eqnarray*}
					Without loss of generality, we can assume
					$\U_p$ connected and $\phi(I^+(p)\cap \U_p) = I^+(p')\cap \U_{p'}$; this way the map
					
						$$\psi_{p'} \circ \phi \circ \psi_p^{-1}: I^+(\psi_p(p))\cap \V_p \rightarrow I^+(\psi_{p'}(p'))\cap \V_{p'}$$
						is an injective $\mass{}$-morphism.
						The future of a BTZ point in $\V_p$ is in the regular part of a neighborhood of some BTZ point of $\V_p$;
						Proposition $\ref{prop:isom_strong}$ applies thus $\alpha=0$ and $\psi_{p'} \circ \phi \circ \psi_p^{-1}$
						is the restriction of an isomorphism $\gamma_p$ of $\mass{0}$.
						
						Consider now a family $(\U_p\xrightarrow{\psi_p} \V_p, \U_{p'} \xrightarrow{\psi_{p'}} \V_{p'}, \gamma_p)_{p\in C}$ of such charts. For all $p\in C$, $\U_p\cap (\pgcd)$ is connected, the $\mass{\geq0}$-morphisms  
$$	
						\fonction
							{\phi_p}
							{(\pgcd) \cup \U_p}
							{M_2}
							{x}
							{
								\left \lbrace
									\begin{array}{ll}
										\phi(x)										& 	~\mathrm{if}~x\in \pgcd\\
										\psi_{p'}^{-1}\circ \gamma \circ \psi_p(x)	& ~\mathrm{if}~x \in \U_p
									\end{array}
								\right.
							}.
					$$				
					are then well defined.
						Each set $(\pgcd )\cup \U_p$ for $p\in C$ is open and connected, furthermore the intersection of two such domains contains $\pgcd$ which is connected therefore: 
						$$ \forall p,q\in C, \quad \left[(\pgcd )\cup \U_p\right] \cap \left[(\pgcd )\cup \U_q\right] \text{ is connected.}$$
						 By analyticity, the $\mass{\geq0}$-morphism
						$$	
						\fonction
							{\overline \phi}
							{(\pgcd) \cup \bigcup_{p\in C}\U_p}
							{M_2}
							{x}
							{
								\phi_p(x) \text{ if } x\in (\pgcd )\cup \U_p
							}
					$$
					is well defined. 
					
					For every $p\in C$ and every
					$q\in \sing_0(\U_p)$ the points $q,\overline \phi(q)$  are unseparated in $\ppcm$.
					Therefore, either $q\in C$ or $q\in \pgcd$;  in any case 
					
					$$(\pgcd)\cup C = (\pgcd)\bigcup_{p\in C} \U_p$$
					so $(\pgcd)\cup C$ is open and $\phi$ extends to $\overline \phi: (\pgcd)\cup C \rightarrow M_2$.
					
				\item Notice that if $(p,q)\in M_1\times M_2$ are unseparated then $I^+(\overline \phi(p))=\phi(I^+(p))=I^+(q)$. Since $M_2$ is globally hyperbolic, by \ref{prop:FZenon_injectif} $M_2$ is future distinguishing and $q=\overline\phi(p)$
						
				\item In order to prove that $(\pgcd) \cup C$ is globally hyperbolic, in view of Lemma \ref{lem:BTZext_caus_conv}, it suffices to prove that $\forall x\in C,\, \sing_0(J^+_{M_1}(x)) \subset (\pgcd) \cup C$.
				
				 Let $p\in C$ and $q\in \sing_0(J^+_{M_1}(p))$. Let $c:[0,1]\rightarrow M_1$ be a future causal curve from $p$ to $q$. By Lemma \ref{lem:past_BTZ}, $c([0,1])\subset \sing_0(M_1)$ and is thus a BTZ-line segment. 
				
Consider $I:= \left\{t_0\in [0,1] ~|~ c([0,t_0])\subset \pgcd \cup C\right\}$;  since $\pgcd\cup C$ is open, $I$ is open. Let $t_0 = \sup I$ and consider a decreasing sequence $(q_n)_{n\in\N} \in M_0^\N$ converging toward $c(t_0)$; notice that $q_n\in J^+_{\pgcd\cup C}(c(0))$ so that $\phi(q_n)\in J^+_{M_2}(\overline\phi\circ c(0))$. Therefore, the sequence $(\phi(q_n))_{n\in\N}$ is decreasing and bounded below in $M_2$; by global hyperbolicity of $M_2$, it converges in $M_2$. Finally,  $c(t_0)\in  \pgcd \cup C$  and $I$ is closed. 
To conclude, $I$ is both open and closed in $[0,1]$ and is thus equal to $[0,1]$; in particular, $c(1)=q \in  \pgcd \cup C$. 

$(\pgcd) \cup C$ is then globally hyperbolic.

				\item From what we proved above,
					$(\pgcd)\cup C$ is a globally hyperbolic sub-BTZ-extension of $M_0\rightarrow M_1$
					endowed with an embedding into $M_2$. 
					By maximality of $\pgcd$, we have $C=\emptyset$ and  $\ppcm$ is thus Hausdorff.
					\item The maps $M_1,M_2\xrightarrow{\pi_1,\pi_2} \ppcm$ are BTZ-embeddings and $M_1,M_2$, are globally hyperbolic; then, by Lemma \ref{lem:BTZext_caus_conv}, $\pi_i(M_i)$ is causally convex in $\ppcm$ for $i\in\{1,2\}$ and, by Corollary \ref{cor:GH_causal}, $\ppcm$ is causal.
					
					Let $p,q\in \ppcm$, observe that if $p,q$ are both in the image of $\pi_1$ or both in the image of $\pi_2$ then by causal convexity and global hyperbolicity of $M_1,M_2$, the diamond $J^+(p)\cap J^-(q)$ is compact. 
					
If either of $p,q$ is in $\reg_{>0}(\ppcm)$ then the previous observation applies, we thus assume that $p,q$ are both in $\sing_0(\ppcm)$ and, without loss of generality, that $p\in \pi_1(M_1)\cap J^-(q)$.  Consider $q'\in I^+(q)\subset \reg_{>0}(\ppcm)\subset \pi_1(M_1)$, we notice that $q\in J^-_{\ppcm}(q')$ thus, by causal convexity of $\pi_1(M_1)$, we have $$q\in J^-_{\ppcm}(q')\cap J^+_{\ppcm}(p) = J^-_{\pi_1(M_1)}(q')\cap J^+_{\pi_1(M_1)}(p).$$
In particular, $q\in \pi_1(M_1)$ and the previous observation applies again.

			\end{itemize}
		\end{proof}

\subsection{Proof of Theorem 2}

We now proceed to the proof of the Maximal BTZ-extension Theorem.
\setcounter{theo}{1}
 \begin{theo}\label{theo:BTZ_ext} Let $M$ be globally hyperbolic $\mass{\geq 0}$-manifolds, there exists a globally hyperbolic BTZ-extension $M\xrightarrow{\iota} N$ which is maximal among such extensions. Moreover, $M\xrightarrow{\iota}N$ is unique up to equivalence. We  call this extension
the maximal BTZ-extension of $M$ and denote it by 
  $\mathrm{Ext}_{BTZ}(M)$.   
\end{theo}

		\begin{proof}
			Define $\overline M:= \varinjlim N$ as the inductive limit \cite{MR1712872} in the category of possibly non-metrizable topological manifolds \cite{MR3244277} where $N$ goes through the BTZ-extensions of $M$. This inductive limit is well defined, since the family of BTZ-extensions of $M$ is right filtered by Proposition \ref{prop:ppcm} and the BTZ-extensions of $M$ all have the same cardinality. 
			
			We endow $\overline M$ with the $\mass{\geq0}$-structure induced by the (open) topological embeddings $N\rightarrow \overline M$ for $N$ BTZ-extension of $M$.
			
			\begin{itemize}
\item An inductive limit $N$ of BTZ-extensions of $M$ which is second countable is globally hyperbolic. Indeed, by Corollary \ref{cor:GH_causal}, $N$ is causal. Furthermore, for any $p,q \in N$, there exists a globally hyperbolic BTZ-extension $M_1$ of $M$  which contains $p$ and $q$. By Lemma \ref{lem:BTZext_caus_conv},  $M_1$ is causally convex in $N$. Finally, $J^+_{N}(p)\cap J^-_{N}(q) = J^+_{M_1}(p)\cap J^-_{M_1}(q) $ which is compact.			
			
				\item
					Let $\Sigma$ be a Cauchy surface of $M$. Each BTZ line of $\overline M\setminus M$ is associated to a puncture of $\Sigma$ ie an end which admits a disc neighborhood. By Richards Theorem \cite{MR143186}, a surface can only have countably many such ends. Therefore, $\sing_0(\overline M)$ has at most countably many connected components.		
				
				 Consider a connected component $\Delta$ of $\sing_0(M)$.  The $\mass{\geq 0}$-structure of $\overline M$ induces a $\left(G ,\sing(\mass{0}) \right)$-structure on $\Delta$ with $G$ the image of $\isom(\mass{0})$ in the group of homeomorphisms of $\sing_0(\mass{0})$. Since the isometries of $\mass{0}$ act by translation on $\sing_0(\mass{0})$, the singular line $\Delta$ is actually endowed with a $(\RR,\RR)$-structure ($\RR$ acting by translation on itself). 
				 Finally, $\Delta$ admits a Riemannian metric and is thus second countable. 
				 
				 There thus exists a countable family $(M_i)_{i\in I}$ of BTZ-extensions  of $M$ whose inductive limit is equal to $\overline M$  and then $\overline M$ is second countable.

			\end{itemize}

   Finally, $\BTZext(M) := \overline M $ has the wanted properties.
		\end{proof}
		
\begin{remark}
 There are other ways to prove the BTZ lines are second countable in the proof above. One can also use the fact that every possibly non metrizable 1-manifolds are type I \cite{MR3244277} to show each BTZ-lines admits a neighborhood which is a type I submanifold of $\overline M$; then use separability of $M$ to conclude using the fact that every type I separable manifolds are metrizable \cite{MR3244277}. 
\end{remark}

\subsection{A complement on Cauchy-completeness}
			\label{sec:complete_stability}
                    
                We notice that the proof of Proposition \ref{prop:abs_max} yields another meaningful result. 
                \begin{propou}\label{prop:reg_complete}
                  Let $M$ be a globally hyperbolic Cauchy-complete Cauchy-maximal $\mass{\geq0}$-manifold. 
                  Then, the complement of the BTZ-lines $\reg_{>0}(M)$ is globally hyperbolic Cauchy-maximal and Cauchy-complete. 
                  \end{propou}
                  \begin{proof}
                   Global hyperbolicity follows from Remark \ref{rem:reg_GH} and the proof of Cauchy-complete\-ness of Proposition \ref{prop:abs_max} applies as is. Continuing along the same route, consider some Cauchy-extension $N$ of $M^*:=\reg_{>0}(M)$, the proof of Proposition \ref{prop:abs_max} shows the $\mass{\geq0}$-manifold $\overline M$ obtained by gluing $M$ to $N$ along $M^*$ using Lemma \ref{lem:extension_hausdorff} is globally hyperbolic. We are left to show $M\rightarrow \overline M$ is a Cauchy-extension.

                   Consider $\Sigma$ a Cauchy-surface of $M$ and $\Sigma^*$ the Cauchy-surface of $M^*$ obtained by applying point (ii) of Lemma \ref{lem:curve_extension} in disjoint spear neighborhoods of $\sing_0(M)$.          
                   Using Corollary \ref{cor:intersect} we deduce that $\Sigma^*$ is a Cauchy surface of $M^*$, thus a Cauchy surface of $N$ and then that  $\Sigma$ is a Cauchy surface of $\overline M$.
                  \end{proof}

                 Considering what has been done so far, a reciproque to Proposition \ref{prop:reg_complete} is within reach. 
                 
				\begin{theo} \label{theo:BTZ_Cauchy-completness}
				Let $M$ be a globally hyperbolic $\mass{\geq0}$-manifold. The following a then equivalent:			
				\begin{enumerate}[(i)]
				\item $\reg_{>0}(M)$ is Cauchy-complete and Cauchy-maximal;
				\item there exists a BTZ-extension of $M$ which is Cau\-chy-com\-plete and Cau\-chy-maxi\-mal;
				\item $\BTZext(M)$ is Cauchy-complete and Cauchy-maxi\-mal.
				\end{enumerate}
				\end{theo}

				\begin{proof}
				  $(iii)\Rightarrow (ii)$  is trivial and $(ii)\Rightarrow (i)$ is a consequence of Proposition \ref{prop:reg_complete}.  Let us prove $(i)\Rightarrow (iii)$.
				  
				  Let $M_0:=\reg_{>0}(M)$ and assume $M_0$ is a globally hyperbolic Cauchy-complete Cauchy-maximal $\mass{>0}$-manifold.
				  Consider, $M_1$ the maximal BTZ-extension of $M_0$ and $M_2$ the Cauchy-maximal extension of $M_1$. Without loss of generality, we assume $M_0\subset M_1\subset M_2$; we wish to prove $M_2=M_1$ and that $M_2$ is Cauchy-complete. To this end, define $M_2^* := \reg_{>0}(M_2)$ the complement of the BTZ-lines in $M_2$ and consider $\Sigma_0$,$\Sigma_1$ spacelike Cauchy-surfaces of $M_0$ and $M_1$ respectively.  
				  
				  \item \underline{Step 1 : every BTZ-line of $M_1$ admits a spear neighborhood.}
				  
				  Let $\Delta$ be BTZ-line of $M_2$ and let $p\in \Delta$ in the past of $\Sigma_1$. By Lemma \ref{lem:tube} there exists $\S\xrightarrow{\phi} \mass{0}$ a spear neighborhood of vertex $p$ then 
				  by Lemma \ref{lem:add_BTZ_param} this spear neighborhood can be chosen in such a way that $\phi(\Sigma \cap \S)$ is a graph of some $C^1$ function $\tau_1:\mathbb D_R\rightarrow \RR$ above the disk $\mathbb D_R$ of some radius $R>0$ in $\mass{0}$. In particular, $\Sigma_1\cap \S$ is homeomorphic to a disk. By Lemma \ref{lem:blunt_spear_neighborhood}, one can reduce the size of $\S$ in such a way $\phi(\Sigma_0\cap \S)$ is the graph of a some function $\tau_0:\mathbb D_R^*\rightarrow \RR$ in $\mass{0}$ with $\tau_0$. Since $\Sigma_0$ is complete, by Lemma \ref{lem:surface_complete} $ \lim_{(\r,\theta)\rightarrow 0}\tau_0(\r,\theta) = +\infty$. We can thus assume $\tau_0: \mathbb D_R\rightarrow \RR_+^*$
				  
				  Now, consider the spacelike circle $\C=\Sigma_0\cap \partial \S\subset M_0$ and let $\C'$ (resp. $\C''$) be a spacelike circle in the shaft of $\S$, in the future (resp. the past) of $\C$  and sufficiently close to $\C$ such that the lightlike cylinder $J^-_{\S}(\C')\cap J^+_{\S}(\C)$ (resp. $J^-_{\S}(\C)\cap J^+_{\S}(\C'')$ ) is a subset of $M_0$. Consider $$N:= (M_0\setminus J^+_{M_0}(\C'))\cup \reg(\mathrm{Int}(\S) \setminus J^-_\S(\C'')),$$ 
				  by Lemma \ref{lem:surface_complete} $\lim_{(\r,\theta)\rightarrow 0}\tau_0=+\infty$ and a simple analysis using Lemma \ref{lem:intersect_number_complete} shows $\Sigma_0$ is a Cauchy-surface of $N$. By Cauchy-maximality of $M_0$, we deduce that $N\subset M_0$.  In particular, $M_0$ contains some blunt spear neighborhood $\S_1^*$ the regular part of a spear neighborhood $\S_1\subset M_2$ around $\Delta$ of some vertex $p'$. Notice $M_1\cup \mathrm{Int}(\S_1)$ is a globally hyperbolic BTZ-extension of $M_1$, by BTZ-maximality of $M_1$, we deduce that $\mathrm{Int}(\S_1)\subset M_1$. Therefore, every BTZ-line of $M_1$ admits a spear neighborhood.
				  
				  \item \underline{Step 2 : $M_1=M_2$ is Cauchy-complete}
				  
				  Since every BTZ-line $\Delta$ of $M_1$ admits a spear neighborhood 
				  $\S_{\Delta}$, in very much the same way as in the proof of Proposition \ref{prop:abs_max}, Lemmas \ref{lem:add_BTZ_param}, \ref{lem:blunt_spear_neighborhood} and \ref{lem:curve_extension}  allow to construct a Cauchy-surface $\Sigma_1$ of $M_1$ by replacing pieces of a complete Cauchy-surface of $M_0$ inside each spear neighborhood by a compact piece cutting the BTZ line. This operation preserve metric completeness and using again Corollary  \ref{cor:intersect} we see that $\Sigma_1$ is a Cauchy-surface, $M_1$ is thus Cauchy-complete.
				  
                  The Cauchy-surfaces $\Sigma_0$ and $\Sigma_1$ now agree on the complement of the spear neighborhoods. Using again Corollary \ref{cor:intersect}, we see that $\Sigma_0$ is a Cauchy-surface of $M_2^*$ so $M_2^*$ is a Cauchy-extension of $M_0$. Since $M_0$ is Cauchy-maximal, then $M_2^*=M_0$. Since $M_2$ is a globally hyperbolic BTZ-extension of $M_2^*$ then, by BTZ-maximality of $M_1$, we deduce that $M_2= M_1$.

				\end{proof}

\section{Proof of Theorem 1}
We can now conclude the proof of the main Theorem.
\setcounter{theo}{0}

\begin{theo}\label{theo:princ} Let $\Sigma^*$ be a surface of genus $g$ with exactly $s$ marked points such that $2-2g-s<0$. The deformation space of globally hyperbolic Cauchy-maximal $\mass{0}$-manifolds homeomorphic to $\Sigma\times \RR$ with exactly $s$ marked singular lines can be identified to the tangent bundle of the Teichmüller space of $\Sigma^*$ .
\end{theo}
\begin{proof} 
Consider $\M_{g,s}(\mass{})$ the deformation space of globally hyperbolic Cau\-chy-complete $\mass{}$-maximal $\mass{}$-manifold marked by $\Sigma^*\times \RR$. 
 By Corol\-la\-ry \ref{cor:common_reg}, the map $\M_{g,s}(\mass{0})\xrightarrow{\reg}\M_{g,s}(\mass{})$ is injective and, as a direct 
 consequence of  Remark 4.19  of 
 \cite{barbot_globally_2004}, the holonomy map $ \M_{g,s}(\mass{})\xrightarrow{~\mathrm{Hol}~} T\T_{g,s}$ is injective. Then so is the composition of the two, so the holonomy map obtained in Proposition \ref{prop:hol_map} is injective. By Proposition \ref{prop:surjective},  it is also surjective.
\end{proof}

\bibliographystyle{alpha}
\bibliography{note}	

\end{document}